\documentclass{amsart}
\usepackage[utf8]{inputenc}    
\usepackage[T1]{fontenc}
\usepackage[english]{babel}     
\usepackage[a4paper, left=1in, right=1in, top=1in, bottom=1in]{geometry}
\usepackage{indentfirst}
\usepackage{comment}
\usepackage{graphicx}           
\usepackage{amsmath,amssymb}    
\usepackage{amsthm}             
\usepackage{mathtools}          
\usepackage{enumitem}           
\usepackage{listings}           
\usepackage{todonotes}          
\usepackage{hyperref}           
\usepackage[cal=boondox]{mathalfa}
\usepackage{relsize}
\usepackage{stmaryrd}
\usepackage{setspace}
\usepackage{chngcntr}

\counterwithin{equation}{subsection}
\counterwithin{figure}{subsection}

\setenumerate[1]{label={(\arabic*)}} 
\setlist[enumerate]{itemsep=0pt}
\setlist[itemize]{itemsep=0pt}

\newcommand\restr[2]{{
  \left.\kern-\nulldelimiterspace 
  #1 
  \right |_{#2} 
  }}

\newtheorem{theorem}{Theorem}[section]
\newtheorem{corollary}{Corollary}[theorem]
\newtheorem{lemma}[theorem]{Lemma}
\newtheorem{proposition}[theorem]{Proposition}
\newtheorem*{claim}{Claim}

\newtheorem{innercustomthm}{Theorem}
\newenvironment{customthm}[1]
  {\renewcommand\theinnercustomthm{#1}\innercustomthm}
  {\endinnercustomthm}

\theoremstyle{definition}
\newtheorem{definition}[theorem]{Definition}
\newtheorem{defprop}[theorem]{Definition/Proposition}

\newtheorem*{remark}{Remark(s)}

\begin{document}

\title{On transversely holomorphic partially hyperbolic flows}
\author{Mounib Abouanass}
\date{\today}

\maketitle

\begin{abstract}
    In this paper, we study transversely holomorphic partially hyperbolic flows, i.e. those whose holonomy pseudo-group is given by biholomorphic maps.
    We prove in the seven-dimensional case that under the assumption that the subcenter distribution is integrable to a flow invariant compact foliation with trivial holonomy, then the flow projects, by a smooth fiber bundle map, to a transversely holomorphic Anosov flow on a smooth five-dimensional manifold which is, in case of topological transitivity, either $C^\infty$ orbit equivalent to the suspension of a hyperbolic automorphism of a complex torus, or, up to finite covers, $C^\infty$-orbit equivalent to the geodesic flow of a compact hyperbolic manifold.
\end{abstract}

\tableofcontents

\section{Introduction}

Smooth flows on smooth manifolds have long been of interest to both mathematicians and physicists. Their approach can be either dynamical (by considering the one-parameter subgroup of smooth diffeomorphisms $(\varphi^t)_{t \in \mathbb R}$, using various tools coming from dynamical systems and ergodic theory), or geometrical (by considering the the partition of the phase space into orbits, i.e. the orbit foliation).

A classical approach in continuous dynamics involves assuming transverse structures for the flow.
For example, Brunella and Ghys (see \cite{brunella_umbilical_1995}, \cite{brunella_transversely_1996}, \cite{ghys_transversely_1996}) have studied transversely holomorphic flows on smooth three-manifolds, that is smooth flows whose orbit foliation $\Phi$ can be given by a smooth foliated atlas whose transition maps are holomorphic.
They achieved a complete classification using advanced topological and analytical techniques, see also \cite{carriere_flots_1984}.
However, their techniques fail for higher dimensions due to several reasons: a priori non-conformity, non-equivalence of harmonic functions and holomorphic maps, and the lack classification of compact complex manifolds of complex dimension three, among others.

On the other hand, Anosov flows exhibit rich dynamical properties. 
For instance, algebraic Anosov flows (that is finite covers of suspensions of Anosov diffeomorphisms or geodesic flows on hyperbolic manifolds) have been studied by Ghys. Their geometric properties are closely linked with the structure of their weak stable and unstable foliations: for example, Plante showed in \cite{plante_anosov_1981} that if the weak stable or weak unstable foliation is transversally affine, then the flow is $C^\infty$-orbit equivalent to the suspension of an Anosov diffeomorphism. 

By combining the rigid properties of holomorphic maps with the geometric richness of Anosov flows, we have achieved a classification of transversely holomorphic Anosov flows (i.e. those whose orbit foliation is transversely holomorphic) on smooth compact connected five-dimensional manifold (see \cite{abouanass_global_2025}).

In recent decades, interest has shifted toward a broader class of systems---\textit{partially hyperbolic} ones---pioneered notably in the works of Hirsch--Pugh--Shub \cite{hirsch_invariant_1977} and Brin--Pesin \cite{brin_partially_1974}. A diffeomorphism $f$ on $M$ is said to be (uniformly) partially hyperbolic if $TM$ admits a splitting into three invariant, transverse bundles $
TM = E^s \oplus E^c \oplus E^u,
$
where $E^s$ is uniformly contracted, $E^u$ uniformly expanded, and the \textit{center bundle} $E^c$ is neither contracted nor expanded as strongly. This class includes, for example, time-$t$ maps of Anosov flows. A smooth flow is called partially hyperbolic when its time-one map satisfies this definition; in this case, the flow direction is always contained in the center distribution.

In \cite{abouanass_dynamical_2026}, we refined this definition for flows: we require that the center distribution further decomposes as
$E^c = \mathbb RX \oplus E^{\hat{c}}$,
where $X$ is the smooth non-vanishing vector field inducing the flow, and $E^{\hat{c}}$ is an invariant subbundle called the \textit{subcenter distribution}. This definition has the advantage of considering explicitly a subbundle transverse to the flow. 
We proved that if the subcenter distribution of a partially hyperbolic flow on a smooth compact connected manifold is integrable to a flow-invariant compact foliation $\mathcal{F}^{\hat{c}}$ with trivial holonomy, then
the flow is \textit{dynamically coherent}: there exist flow-invariant foliations $\mathcal{F}^{\hat{c}s}$ and $\mathcal{F}^{\hat{c}u}$ tangent to $E^{\hat{c}s} = E^s \oplus E^{\hat{c}}$ and $E^{\hat{c}u} = E^u \oplus E^{\hat{c}}$, respectively.
In fact, we proved that the subcenter foliation is \textit{complete}, i.e.
$
\bigcup_{z \in \mathcal{F}^{\hat{c}}(x)} \mathcal{F}^{*}(z) = \bigcup_{w \in \mathcal{F}^{*}(x)} \mathcal{F}^{\hat{c}}(w)
$
for all $x \in M$ and $* \in \{s,u\}$. 
By Epstein \cite{epstein_foliations_1976}, the fact that the subcenter foliation $\mathcal{F}^{\hat{c}}$ is compact with trivial holonomy ensures that the leaf space $M/\mathcal{F}^{\hat{c}}$ is a topological manifold.
If a partially hyperbolic flow has a complete subcenter foliation, then the corresponding time-one map is a dynamically coherent partially hyperbolic diffeomorphism with a complete center foliation (see \ref{lem:weakfol}).

In the discrete case, partially hyperbolic diffeomorphisms which are quasiconformal, and in particular holomorphic partially hyperbolic diffeomorphisms, have been studied in \cite{butler_uniformly_2018}; \cite{xu_holomorphic_2025} and it has been shown in complex dimension three that if the center distribution is tangent to an invariant compact foliation with trivial holonomy, then the latter is holomorphic and the system is just a holomorphic skew product over a linear automorphism on a complex 2-torus.

The aim of this article is to prove an analog statement for transversely holomorphic partially hyperbolic flows in dimension seven. The main result is:

\begin{customthm}{A}
    Let $(\varphi^t)$ a smooth transversely holomorphic partially hyperbolic flow on a smooth compact manifold $M$ of dimension $7$ with a flow invariant compact subcenter foliation $\mathcal{F}^{\hat{c}}$ with $C^1$ leaves and trivial holonomy.\\
    Then 
    \begin{enumerate}[label=(\roman*)]
        \item The center-stable $\mathcal{F}^{cs}$, center-unstable $\mathcal{F}^{cu}$ and center $\mathcal{F}^{c}$ foliations are transversely holomorphic and $C^\infty$, and the subcenter foliation $\mathcal{F}^{\hat{c}}$ is $C^\infty$;
        \item There exists a smooth compact connected manifold $N$ of dimension $5$, a map $p:M\to N$ defining a $C^\infty$ fiber bundle whose fibers are exactly the subcenter leaves, and a smooth transversely holomorphic Anosov flow $(\psi^t)$ on $N$ such that for every $t\in \mathbb R$, $p\circ \varphi^t=\psi^t\circ p$ ;
        \item In particular if $(\psi^t)$ is topologically transitive, then $(\psi^t)$ is either $C^\infty$-orbit equivalent to the suspension of a hyperbolic automorphism of a complex torus, or, up to finite covers, $C^\infty$-orbit equivalent to the geodesic flow on the unit tangent bundle of a compact hyperbolic three-dimensional manifold.
    \end{enumerate}
\end{customthm}

In section \ref{sec:4}, we prove some fundamental results such as the fact that the stable, unstable and subcenter leaves of a transversely holomorphic partially hyperbolic flow have complex structures which are preserved by the action of the flow. This is also true for the subcenter-stable and subcenter-unstable foliations if they exist.
This leads, thanks in particular to the (holomorphic) non-stationary linearization of $\mathcal{F}^u$, to the quasiconformality of the center-stable holonomy as well as a form of Fubini's theorem in the seven-dimensional case.

In section \ref{sec:5}, we prove the most important result, that is the holomorphicity of the subcenter foliation inside subcenter-stable and subcenter-unstable leaves, which implies the main theorem in section \ref{sec:6}.

\section{Reminders and definitions}
We only consider smooth ($C^\infty$) connected manifolds without boundary.

\subsection{Foliations}
We start by recalling some notions coming from the theory of foliations (see \cite{camacho_geometric_1985}, \cite{moerdijk_introduction_2003}, , \cite{lee_manifolds_2009} and \cite{candel_foliations_2000} for more details). 
\begin{defprop}
\label{defprop:fol}
    Let $M$ a smooth manifold of dimension $n\in \mathbb N$. Let $r \geq 1$ (or $r\in \{\infty\}$) and $k\leq n$.\\
    A \emph{$C^r$ foliation} $\mathcal{F}$ of dimension $k$ on $M$ (or codimension $n-k$) is defined by one of the following equivalent assertions:
    \begin{enumerate}
        \item a $C^r$-atlas $(U_i, \psi_i)_i$ on $M$ which is maximal with respect to the following properties:
            \begin{enumerate}
                \item For all $i$, $\psi_i(U_i)=U_i^1 \times U_i^2$, where $U_i^1$ and $U_i^2$ are connected open subsets of $\mathbb R^k$ and $\mathbb R^{n-k}$ respectively ;
                \item For all $i,j$, there exist $C^r$ maps $f_{ij}$ and $h_{ij}$ such that
                \[\forall(x,y) \in \psi_i(U_i \cap U_j) \subset \mathbb R^k \times \mathbb R^{n-k}, \; \psi_i \circ \psi_j^{-1}(x,y)=(f_{ij}(x,y), h_{ij}(y)).\]
            \end{enumerate}
        \item a maximal atlas $(U_i, s_i)_i$, where each $U_i$ is an open subset of $M$ and $s_i:U_i \to \mathbb R^{n-k}$ is a $C^r$ submersion, satisfying:
            \begin{enumerate}
                \item $\bigcup_i U_i = M$ ;
                \item For all $i,j$, there exists a $C^r$ diffeomorphism
                \[\gamma_{ij}: s_i(U_i\cap U_j) \to s_j(U_i\cap U_j) \text{  such that  } \restr{s_i}{U_i\cap U_j} = \restr{\gamma_{ij}\circ s_j}{U_i\cap U_j}. \]
            \end{enumerate}
        \item a partition of $M$ into a family of disjoint connected $C^r$ immersed submanifolds $(L_\alpha)_\alpha$ of dimension $k$  such that for every $x \in M$, there is a $C^r$ chart $(U, \psi)$ at $x$, of the form $\psi: U \to U^1 \times U^2$ where $U^1$ and $U^2$ are connected open subsets of $\mathbb R^k$ and $\mathbb R^{n-k}$ respectively, satisfying: for each $L_\alpha$, for each connected component $(U \cap L_\alpha)_\beta$ of $U\cap L_\alpha$, there exists $c_{\alpha,\beta} \in \mathbb R^{n-k}$ such that 
        \[ \psi ((U\cap L_\alpha)_\beta) = U^1 \times \{c_{\alpha, \beta}\}.\]
     \end{enumerate}
    The maps $h_{ij}$ and $\gamma_{ij}$ are called \textit{transition maps}.\\
    We will call (for each equivalent assertion $(1)$, $(2)$ and $(3)$):
    \begin{itemize}
        \item a \emph{plaque}:
            \begin{enumerate}
                \item a set of the form $\psi_i^{-1}(U_i \times \{c\})$, for $c \in U_i^2$ ;
                \item a connected component of a fiber of $s_i$ ;
                \item a connected component $(U\cap L_\alpha)_\beta$ of $U\cap L_\alpha$.
            \end{enumerate}
        \item a \emph{leaf} of the foliation an equivalence class for the following equivalence relation: two points $x$ and $y$ are equivalent if and only if there exist a sequence of foliation charts $U_1, \ldots, U_k$ and a sequence of points $x=p_0, p_1, \ldots, p_{k-1},p_k=y$ such that, for $i\in\llbracket 1,k\rrbracket$, $p_{i-1}$ and $p_{i}$ lie on the same plaque of $U_i$.
    \end{itemize}
    \end{defprop}
    We will represent abusively a foliation by the set of its leaves $\mathcal{F}$ and will note $\mathcal{F}_x$ the leaf of the foliation containing $x\in M$. \\
    In case of the orbit foliation of a non-vanishing smooth vector field, we will write $\Phi$. 
    \begin{remark}
    \label{rem:fol1}
    \leavevmode
    \begin{itemize}
        \item A $C^r$ foliation $\mathcal{F}$ gives rise to a $C^{r-1}$ subbundle of $TM$, noted $T\mathcal{F}$, called the \textit{tangent bundle} to the foliation $\mathcal{F}$, whose fibers are the tangent spaces to the leaves. We also define the \textit{normal bundle} $\nu:= TM \diagup T\mathcal{F}$ of the foliation. It is a vector bundle whose transition functions are given by the differential of the transition maps $h_{ij}$.
        \item We can define in the same way a \emph{holomorphic} foliation on a complex manifold by replacing "$C^r$" in the above definition with "holomorphic" and $\mathbb R$ with $\mathbb C$.
    \end{itemize}
    \end{remark}

From the first and third formulations we can define a more general notion of foliation adapted to our purposes: 
\begin{definition}
    Let $M$ a smooth manifold of dimension $n\in \mathbb N$. Let $r \geq 1$ (or $r\in \{\infty\}$) and $k\leq n$.\\
    A \emph{(topological) foliation} with $C^r$ leaves, $\mathcal{F}$, of dimension $k$ (or codimension $n-k$) on $M$ is defined by one of the following equivalent assertions:
    \begin{enumerate}
        \item a $C^0$-atlas $(U_i, \psi_i)_i$ on $M$ which is maximal with respect to the following properties:
            \begin{enumerate}
                \item For all $i$, $\psi_i(U_i)=U_i^1 \times U_i^2$, where $U_i^1$ and $U_i^2$ are connected open subsets of $\mathbb R^k$ and $\mathbb R^{n-k}$ respectively ;
                \item For all $i,j$, there exist $C^0$ maps $f_{ij}$ and $h_{ij}$ such that
                \[\forall(x,y) \in \psi_j(U_i \cap U_j) \subset \mathbb R^k \times \mathbb R^{n-k}, \; \psi_i \circ \psi_j^{-1}(x,y)=(f_{ij}(x,y), h_{ij}(y))\]
                and for every such $y$, the map $x\mapsto f_{ij}(x,y)$ is $C^r$.
            \end{enumerate}
        \item a partition of $M$ into a family of disjoint connected $C^r$ immersed submanifolds $(L_\alpha)_\alpha$ of dimension $k$  such that for every $x \in M$, there is a $C^0$ chart $(U, \psi)$ at $x$, of the form $\psi: U \to U^1 \times U^2$ where $U^1$ and $U^2$ are connected open subsets of $\mathbb R^k$ and $\mathbb R^{n-k}$ respectively, satisfying: for each $L_\alpha$, for each connected component $(U \cap L_\alpha)_\beta$ of $U\cap L_\alpha$, there exists $c_{\alpha,\beta} \in \mathbb R^{n-k}$ such that 
        \[ \psi ((U\cap L_\alpha)_\beta) = U^1 \times \{c_{\alpha, \beta}\}\]
        and the map $\operatorname{pr}_1\circ \psi|_{(U \cap L_\alpha)_\beta}:(U \cap L_\alpha)_\beta\to U^1$ is a $C^r$ chart for $L$.
     \end{enumerate}
     We call such an atlas a \emph{foliated atlas} adapted to $\mathcal{F}$. 
\end{definition}

\begin{remark}
\label{rem:fol2}
    \leavevmode
        \begin{itemize}
        \item For every $y$, the map $x\mapsto f_{ij}(x,y)$ in the previous definition is automatically a $C^r$ diffeomorphism.  
        \item We can still define the tangent bundle $T\mathcal{F}$ to the foliation $\mathcal{F}$ if it is a foliation with $C^r$ leaves, $r\geq 1$.
        It is a vector bundle over $M$ whose transition maps are given by the differential along the first variable of the maps $f_{ij}$, ie by $d_1f_{ij}$.
        \item In case $T\mathcal{F}$ is a continuous subbundle of $TM$, the foliation $\mathcal{F}$ is said to be \emph{integral} (see \cite{wilkinson_dynamical_2008} and \cite{pugh_holder_1997} for a discussion about integrability of continuous subbundles of $TM$).
        In order to facilitate the statement of some results, we will call integral foliations \emph{$C^0$ foliations}.\\
        Note that every $C^r$ foliation, $r\geq 1$, is integral.
        \item We can define in the same way (if $k=2k'$ is even) a \emph{foliation with holomorphic leaves} by replacing "$C^r$" in the above definition with "holomorphic".
        \item A useful way to study a foliation with $C^r$ leaves locally is to lift it to $TM$. We explain this idea (see \cite{pugh_holder_1997}). 
We fix a smooth Riemannian metric on $TM$ and a point $p\in M$. The exponential map $\exp_p$ is a $C^\infty$ diffeomorphism from an open neighborhood $V_p$ of $0$ in $T_p M$ to an open neighborhood $U_p$ of $p$ in $M$. Therefore, it defines a foliation $\overline{\mathcal{F}}=\exp_p^{-1}(\mathcal{F})$ of $V_p$.
The foliation $\overline{\mathcal{F}}$ has the same regular properties as $\mathcal{F}|_{U_p}$.\\ 
Denote by $F_p$ the tangent space $T_p\mathcal{F}_p$ to the foliation $\mathcal{F}$ at $p$, and by $F_p(\delta)$ the ball at $0$ of size $\delta>0$ in $F_p$.
By restricting $V_p$ if necessary, there exists $\delta>0$ such that the plaques of $\overline{\mathcal{F}}$ are represented by graphs of $C^r$ maps $g_y:F_p(\delta)\to F_p^\perp$, where $y \in F_p^\perp$, such that $g_y(0)=y$.
Remark that the tangent space at $(x,g_y(x)) \in V_p$ of the foliation $\overline{\mathcal{F}}$ is given by the graph of the linear map $d_xg_y$.\\
As a result the map 
$$(x,y)\in V_p \to \exp_p(x,g_y(x))\in U_p$$
defines a foliated chart for $\mathcal{F}$ centered at $p$.

        \end{itemize}
\end{remark}

One interest of integral foliations is that their leaves are uniformly $C^1$ (see \cite{pugh_holder_1997}):
\begin{definition}
    A foliation $\mathcal{F}$ with $C^r$ leaves ($r \geq 1$ or $r=\infty$) has \emph{uniformly $C^r$ leaves} if there exists a foliated atlas $(U_i,\psi_i)_i$ adapted to $\mathcal{F}$ such that for every $i,j$, the partial derivatives of order $k\leq r$ of $x\mapsto f_{ij}(x,y)$ depend continuously on $(x,y)$.
\end{definition}
\begin{remark}
    A foliation $\mathcal{F}$ with $C^r$ leaves has uniformly $C^r$ leaves if and only if for every $p \in M$, the partial derivatives of order $k\leq n$ with respect to $x$ of the map $g=(x,y)\mapsto g_y(x)$ of the above construction depend continuously on $(x,y)$.
\end{remark}
\begin{proposition}
\label{prop:unifleaves}
    Let $\mathcal{F}$ a topological foliation with $C^1$ leaves.\\
    Then $\mathcal{F}$ has uniformly $C^1$ leaves if and only if it is integral.\\
    Any foliation with holomorphic leaves is integral and has uniformly $C^\infty$ leaves.
\end{proposition}
\begin{proof}
    The first point is an immediate consequence of the remark above.\\
    As for the second point, by definition for a foliation $\mathcal{F}$ with holomorphic leaves, the maps $f_{ij}$ depend continuously on $(x,y)$ and are holomorphic with respect to $x$. By analytic properties of holomorphic maps, this implies that $d_1f_{ij}$ depends continuously on $(x,y)$ so $T\mathcal{F}$ is a continuous subbundle of $TM$. This also proves that for every $k$, the partial derivatives of order $k$ of $x\mapsto f_{ij}(x,y)$ depend continuously on $(x,y)$, which proves that the leaves of $\mathcal{F}$ are uniformly $C^\infty$.
\end{proof}
\label{rem:bohnetmetric}
As mentioned in \cite{bohnet_partially_2014}, an integral foliation $\mathcal{F}^*$ on a smooth compact manifold $M$ has the following properties. Denote by $d^*$ the distance in a leaf $L$ of $\mathcal{F}^*$ coming from the induced metric of $M$ on $L$. Fix a finite foliated atlas $(U_i,\psi_i)$ for $\mathcal{F}$. Then:
\begin{itemize}
    \item For every $\eta > 0$, there is $\mu$ so that, if $x,y$ are points in the same plaque such that $d(x,y) < \mu$, then and $d^{*}(x,y) < (1+\eta)d(x,y)$. As a consequence, there exists $\Delta^{*} > 0$ such that for every $i$ and every $x,y \in U_{i}$, if $d^{*}(x,y) < \Delta^{*}$, then $x$ and $y$ are in the same plaque.
    \item If $\mathcal{F}^{*}$ is a compact foliation (that is, every leaf is compact), then each of its leaf intersects any foliation chart in a finite number of plaques (see \cite{camacho_geometric_1985}). Therefore, two points in the same leaf which are close enough in the ambient manifold are in the same plaque. In that case, the above can be reformulated as:\\
    For every $\eta > 0$, there is $\mu > 0$ so that, if $x,y$ are points in the same plaque such that $d(x,y) < \mu$, then $d^{*}(x,y) < (1+\eta)d(x,y)$.
\end{itemize}

From now on, we will always consider $C^0$ foliations on a smooth manifold $M$.

\subsection{Holonomy}
    
We now define the fundamental concept of \textit{holonomy} for an integral foliation (see \cite{moerdijk_introduction_2003}, \cite{camacho_geometric_1985}).
In the following, a $C^0$ diffeomorphism means a homeomorphism. 

\begin{definition}
    A \emph{transversal} to the (integral) foliation $\mathcal{F}$ is a smooth submanifold $T$ of $M$ such that for every $x\in T$, $T_xT\oplus T_x\mathcal{F}=T_xM$.
\end{definition}
\begin{defprop}
\label{defprop:hol}
    Let $ \mathcal{F}$ a $C^r$ foliation on a smooth manifold $M$ ($r \in \mathbb N \cup \{\infty\}$). Let $x,y$ belonging to the same leaf $L$ of $\mathcal{F}$ and $\alpha$ a path from $x$ to $y$ in $L$. Let also $T,S$ two small transversal to $\mathcal{F}$ at $x$ and $y$ respectively.
    \begin{itemize}
        \item If there exists a foliation chart $U$ containing $\alpha([0,1])$, we can define a germ of $C^r$ diffeomorphism $h(\alpha)^{S,T}$ from a small open subset $A$ of $T$ to an open subset of $S$ such that: \begin{enumerate}
            \item $h(\alpha)^{S,T}(x)=y$ ;
            \item For any $x'\in A$, $h(\alpha)^{S,T}(x')$ lies on the same plaque in $U$ as $x'$.
        \end{enumerate}
        Moreover, the germ $h(\alpha)^{S,T}$ does not depend on $U$ nor the path in $L\cap U$ connecting $x$ and $y$.
        \item In the general case, we can choose a sequence of foliation charts $U_1, \ldots, U_k$ such that for all $i$, $\alpha([\frac{i-1}{k}, \frac{i}{k}]) \subset U_i$, and a sequence $(T_i)_i$ of transversal sections of $\mathcal{F}$ at $\alpha(\frac{i}{k})$, with $T_0=T$ and $T_k=S$. Let $\alpha_i$ a path in $L\cap U_i$ from $\alpha(\frac{i-1}{k})$ to $\alpha(\frac{i}{k})$. \\
        We define 
        \[h(\alpha)^{S,T}:=h(\alpha_k)^{T_k, T_{k-1}} \circ \cdots \circ h(\alpha_1)^{T_1, T_0}.
        \]
        Moreover, $h(\alpha)^{S,T}$ depends only on $T, S$ and $\alpha$, and satisfies the same properties $(1), (2)$.
    \end{itemize}
    The germ of $C^r$ diffeomorphism $h(\alpha)^{S,T}$ is called the \emph{$\mathcal{F}$-holonomy} of the path $\alpha$ with respect to the transversals $T$ and $S$.
\end{defprop} 
    We list some basic properties of holonomy maps:
    \begin{enumerate}[label=(\roman*)]
        \item If $\alpha$ is a path in $L$ from $x$ to $y$ and $\beta$ a path in $L$ from $y$ to $z$, and if $T$, $S$ and $R$ are transversal sections of $\mathcal{F}$ at $x$, $y$ and $z$ respectively, then
        \[h(\beta \alpha)^{R,T}=h(\beta)^{R,S} \circ h(\alpha)^{S,T},\]
        where $\beta\alpha$ is the concatenation of the paths $\alpha$ and $\beta$ ;
        \item If $\alpha$ and $\beta$ are homotopic paths in $L$ relatively to endpoints from $x$ to $y$, and if $T$ and $S$ are transversals sections at $x$ and $y$ respectively, then $h(\beta)^{S,T}= h(\alpha)^{S,T}$.
        \item If $\alpha$ is a path in $L$ from $x$ to $y$, and if $T,T'$ and $S,S'$ are pairs of transverse sections at $x$ and $y$, respectively, then \[h(\alpha)^{S',T'}=h(\overline{y})^{S',S} \circ h(\alpha)^{S,T} \circ h(\overline{x})^{T,T'},\] where $\overline{x}$ is the constant path with image $x$.
    \end{enumerate}
    \begin{remark}
    \label{rem:2.2}
    \begin{itemize}
    \leavevmode
        \item The construction of the holonomy diffeomorphism for two points in the same plaque of the same foliation chart $(U, \psi)$ is as follows. Keep the notations of the previous Definition/Proposition. Denote by $n$ the dimension of $M$ and $q$ the codimension of $\mathcal{F}$. 
        $\psi(T) , \psi(S)\subset \mathbb R^{n}$ are vertical graphs of $C^r$ maps from $\mathbb R^q$ to $\mathbb R^{n-q}$.
        By projecting these graphs on $\mathbb R^q$, we obtain two $C^r$ diffeomorphisms $\Psi: T \to \mathbb R^q$ and $\Psi': S \to \mathbb R^q$. Since the leaves of the foliation are the horizontal lines, it comes 
        \[h(\alpha)^{S,T}=\Psi'^{-1} \circ \Psi.\]
        This means that under these coordinates, the holonomy between two points in the same plaque of the same foliation chart is equal to the identity.
        \item The holonomy of a path in a leaf with respect to the natural transversals given by the foliation charts can be expressed in local coordinates as a composition of transition functions $h_{ij}$ (see Definition/Proposition \ref{defprop:fol}). More precisely, let $\alpha$ a path in $L$ from $x$ to $y$ and a sequence of foliation charts $U_1, \ldots, U_k$ such that for all $i$, $\alpha([\frac{i-1}{k}, \frac{i}{k}]) \subset U_i$ and $U_1$ and $U_k$ are centered at $x$ and $y$ respectively.
        Let $T=\psi_1^{-1}(\{0\} \times U_1^2)$ and $S=\psi_k^{-1}(\{0\} \times U_k^2)$.
        Then there exists a small open subset $V$ of $T$ such that 
        \[h(\alpha)^{S,T}= \restr{\psi_k^{-1} \circ h_{k,k-1} \circ \cdots \circ h_{2,1} \circ \psi_1}{V}.
        \]
    \end{itemize}
    \end{remark}
    
\begin{definition}
    Given a $C^r$ foliation $\mathcal{F}$ on a manifold $M$, the set of all holonomy maps of any path with respect to any transversals to $\mathcal{F}$ defines a pseudo-group called the \emph{holonomy pseudo-group} of the foliation $\mathcal{F}$.\\
    For every leaf $L$ of $\mathcal{F}$, $x\in L$ and small transversal section $T$ at $x$, there is a group homomorphism 
    \[h^T:=h^{T,T}: \pi_1(L,x) \to \text{Diff}^r_x(T), \]
    where $\text{Diff}^r_x(T)$ is the group of $C^r$ diffeomorphisms of $T$ which fix $x$, 
    called the \emph{holonomy homomorphism} of $L$, determined up to conjugation.\\
    The set $h^T(\pi_1(L,x))$ is a group called the \textit{holonomy group} of the leaf $L$, determined also up to conjugation.\\
    We define also the space $\widetilde{L}\diagup \ker(h^T)$, where $\widetilde{L}$ is the universal covering space of $L$. It does not depend on $x \in L$ nor $T$ and is called the \emph{holonomy covering} of $L$.
    \end{definition}
    
    We say that a $C^r$ foliation $\mathcal{F}$ has \textit{trivial holonomy} if the holonomy group of every leaf of $\mathcal{F}$ is trivial.

    \begin{remark}
        If we want to study the holonomy of a $C^0$ foliation in a neighborhood of some point $p$, we can lift it to $TM$ as been explained above.
        By the previous notations, the holonomy between two points $z_0=(0,g_{y_0}(0))=(0,y_0)$ and $z=(x,g_{y_0}(x))$ in the same plaque of $\overline{\mathcal{F}}$ with respect to the transversals $\{0\} \times F_p^\perp$ and $\{x\} \times F_p^\perp$ is given by the map 
        $$y=g_y(0)\in B(y_0,\delta_0) \mapsto g_y(x)\in F_p^\perp$$
        where $B(y_0,\delta_0)$ is a small ball of radius $\delta_0$ in $F_p^\perp$ centered at $y_0$.
    \end{remark}

    Even though the holonomy of a $C^0$ foliation $\mathcal{F}$ is only continuous by definition, it can have stronger regularity, such as being $C^r$ (we say that $\mathcal{F}$ \emph{has $C^r$ holonomy}) or even holomorphic. In the latter case, $\mathcal{F}$ is said to be \emph{transversely holomorphic}.\\
    However, as has been discussed in \cite{pugh_holder_1997}, it is not sufficient for a $C^0$ foliation with $C^1$ leaves and $C^1$ holonomy to be a $C^1$ foliation.
    \begin{definition}
        Let $\mathcal{F}$ a $C^0$ foliation of codimension $q$ on a smooth manifold $M$ with $C^r$ holonomy, $r\in \mathbb N^*\cup \{\infty\}$.\\
        We say that \emph{$\mathcal{F}$ has uniformly $C^r$ holonomy} (or the holonomy of \emph{$\mathcal{F}$ is uniformly $C^r$}) if for every pair of points $x,x'$ in the same plaque, the partial derivatives of order $k\leq r$ of $h_{x,x'}^{T',T}(y)$ vary continuously with $(x,y)$.
    \end{definition}

    \begin{remark}
        Following on from the previous remarks:
            A $C^0$ foliation has uniformly $C^r$ holonomy if and only if for every $p \in M$, the partial derivatives of order $k\leq r$ of the map $y\mapsto g_y(x)$ depend continuously on $(x,y)$.
    \end{remark}
    The main interest of the definitions of having uniformly smooth leaves and uniformly smooth holonomy comes from Journé's Lemma (see \cite{journe_regularity_1988}) which can be used to prove the following result (see \cite{pugh_holder_1997}):
    \begin{proposition}
    \label{prop:journe}
        A foliation with uniformly $C^1$ (respectively $C^\infty$) leaves and uniformly $C^1$ (respectively $C^\infty$) holonomy is $C^1$ (respectively $C^\infty$).
    \end{proposition}
    The proof of this result is just an application of Journé's lemma to the map $(x,y) \mapsto g_y(x)$ of the previous remarks.

\subsection{Compact foliations}
We recall some results about compact foliations, that is foliations whose leaves are compact (see \cite{epstein_foliations_1976}, \cite{candel_foliations_2000}, \cite{carrasco_compact_2015}, \cite{bohnet_partially_2014} for more details).
The following is a result of Epstein:

\begin{theorem}
\label{thm:epstein}
    Let $\mathcal{F}$ a compact $C^0$ foliation on a smooth manifold $M$.\\
    Then the following properties are equivalent:
    \begin{enumerate}[label=(\roman*)]
        \item The quotient map $\pi: M\to M/\mathcal{F}$ is closed.
        \item $M/\mathcal{F}$ is Hausdorff.
        \item Every leaf of $\mathcal{F}$ has arbitrarily small saturated neighborhoods.
        \item The saturation of every compact subset of $M$ is compact.
    \end{enumerate}
    If $M$ is compact, then the above properties are also equivalent to:
    \begin{enumerate}[label=(\roman*),start=5]
        \item The holonomy group of every leaf of $\mathcal{F}$ is finite.
        \item The volume of the leaves is uniformly bounded from above.
    \end{enumerate}
    Moreover if $\mathcal{F}$ has trivial holonomy, then the resulting leaf space is a topological manifold.
\end{theorem}
See a proof in \cite{carrasco_compact_2015}. In fact, the equivalence with the two last assertions is due to the following theorem:
\begin{theorem}[Generalized Reeb Stability]
\label{thm:reeb}
    Let $L$ a compact leaf of a $C^r$ foliation on a smooth manifold $M$. Assume its holonomy group $\operatorname{Hol}(L)$ is finite.\\
    Then there exists a normal neighborhood $p:V\to L$ of $L$ in $M$ which is a $C^r$ fiber bundle with structure group $\operatorname{Hol}(L)$.\\
    Furthermore, each leaf $L'|_V$ is a covering space $p|_{L'}:L' \to L$ with $k\leq |\operatorname{Hol}(L)|$ sheets and the leaf $L'$ has a finite holonomy group of order $\frac{\operatorname{Hol}(L)}{k}$.
\end{theorem}
A corollary of this result, and which motivates our assumption of trivial holonomy for the subcenter foliation, is that the leaves of a compact foliation with trivial holonomy are \emph{generic} (see \cite{candel_foliations_2000}):
\begin{corollary}
    Let $\mathcal{F}$ a compact $C^0$ foliation.\\
    Then the set of leaves with trivial holonomy is open and dense.
\end{corollary}

\subsection{Almost complex structure}
\begin{definition}
    Let $M$ a smooth manifold of even dimension $2n\geq2$.\\
    An \emph{almost complex structure} on $M$ is a smooth vector bundle endomorphism $J:TM \to TM$ such that $J^2=-\operatorname{Id}$.
\end{definition}
Every complex manifold, of complex dimension $n$, carries a \emph{canonical} almost complex structure given in local coordinates $(z_1,\cdots, z_n)$, where $z_k=x_k+iy_k$ for $k\in \llbracket 1,n \rrbracket$, by
$$J\left ( \frac{\partial }{\partial x_k}\right) = \frac{\partial }{\partial y_k} \quad\text{ ; } \quad J\left ( \frac{\partial }{\partial y_k}\right) =  -\frac{\partial }{\partial x_k}.$$

\begin{definition}
    An almost complex structure $J$ on a smooth compact manifold $M$ is \emph{integrable} if there exists a complex manifold structure on $M$ compatible with its smooth structure whose canonical induced almost complex structure is $J$.
\end{definition}
Not every almost complex structure is integrable. This can be seen using the following fundamental theorem due to Newlander-Nirenberg (see \cite{kobayashi_foundations_1996}):

\begin{theorem}
    An almost complex structure $J$ on a smooth manifold $M$ is integrable if and only if its \emph{Nijenhuis tensor} 
    $$N_J: (X,Y)\mapsto  [JX, JY] - J[X, JY] - J[J X,Y] -[X, Y]$$ 
    vanishes identically.
\end{theorem}

\subsection{Transversely holomorphic foliations}
We all that in mind, we are able to define (and characterize) \emph{transversely holomorphic foliations} (see \cite{gomez-mont_transversal_1980}, \cite{duchamp_deformation_1979}, \cite{haefliger_deformations_1983},  \cite{brunella_transversely_1996}):
\begin{definition}
\label{def:trholo}
    A $C^r$ foliation $\mathcal{F}$ ($r \in \mathbb N \cup \{\infty\}$) of dimension $k \in \mathbb N$ on a smooth manifold of dimension $k+2q$ ($q\in \mathbb N^*)$ is \emph{transversely holomorphic} if there exists a $C^r$ atlas of foliated charts compatible with the one defining $\mathcal{F}$ whose transition maps $h_{ij}$ are holomorphic.\\
    If $r\geq 1$, such a structure can be given by a maximal $C^r$ atlas $(U_i,s_i)_i$ of submersions compatible with the one defining $\mathcal{F}$ and whose transition maps $\gamma_{ij}$ are holomorphic.
\end{definition}
\begin{remark}
    It does not seem a priori that these two properties are equivalent. We give a sketch of proof of why they are. Firstly, we recall that the $s_i$ and $\gamma_{ij}$ are obtained more or less directly from the foliated charts $\psi_i$ ($s_i=\text{pr}_2 \circ \psi_i$ and $\gamma_{ij}=\text{pr}_2 \circ (\psi_i\circ \psi_j^{-1})$) so that a direction is quite straightforward. The other direction is however a little subtle (see \cite{camacho_geometric_1985}): the $\psi_i$ are obtained from the constant rank theorem applied to the $s_i$, on which they act as the projection onto the second factor (i.e. $s_i \circ \psi_i^{-1} = \text{pr}_2$). Then the $h_{ij}$ are restrictions of the $\gamma_{ij}$.
\end{remark}
The following result proven in \cite{abouanass_global_2025} is a characterization of transversely holomorphic $C^r$ foliations:

\begin{proposition}
\label{prop:caractrholo}
    Let $M$ a smooth manifold of dimension $k+2q$ $(k \in \mathbb N, q \in \mathbb N^*)$ and $\mathcal{F}$ a $C^r$ foliation ($r \in \mathbb N \cup \{\infty\}$) of dimension $k$ on $M$.\\
    The following assertions are equivalent:
    \begin{enumerate}[label=(\roman*)]
        \item $\mathcal{F}$ is transversely holomorphic;
        \item For any points $x, y$ on the same leaf of $\mathcal{F}$, there exist (hence for all) small transversals to the foliation $T$ and $T'$ at $x$ and $y$ respectively, there exist diffeomorphisms $\Psi: T\to V$ and $\Psi':T'\to V'$, where $V, V'$ are open sets of $\mathbb C^q$, such that, for any $\mathcal{F}$-path $\alpha$ joining $x$ and $y$, 
        \[\Psi' \circ h(\alpha)^{T',T} \circ \Psi^{-1}\] is holomorphic.
    \end{enumerate}
    For $r\geq1$, these assertions are also equivalent to:
    \begin{enumerate}[label=(\roman*)]
        \setcounter{enumi}{2}
        \item There exists an almost complex structure on the normal bundle $ \nu=TM \diagup T\mathcal{F}$ of the foliation (that is a $C^r$ vector bundle endomorphism $I:\nu \to \nu$ satisfying $ I^2=-\operatorname{Id}$) such that
        \begin{enumerate}
            \item For any transversal $S$ to the foliation, the almost complex structure $I_S$ induced by $I$ on the tangent bundle $TS \cong \restr{\nu}{S}$ of $S$ is integrable,
            \item $I$ is invariant along $\mathcal{F}$-holonomy maps, that is:\\
            For every $x \in M$, $y \in \mathcal{F}_x$, for every path $\alpha$ in $\mathcal{F}_x$ connecting $x$ to $y$, for every small transversals $T$ and $S$ to $\mathcal{F}$ at $x$ and $y$ respectively,
            $$(h_\alpha^{S,T})^*I_S:= (Dh_\alpha^{S,T})^{-1}\circ I_S\circ Dh_\alpha^{S,T}= I_T .$$ 
        \end{enumerate}
    \end{enumerate}
\end{proposition}

Therefore, in the same spirit as Proposition \ref{prop:unifleaves} :
\begin{proposition}
\label{prop:unifholo}
Any transversely holomorphic foliation has uniformly $C^\infty$ holonomy.
\end{proposition}

    If one studies foliations with holomorphic leaves on a complex manifold, then by analytic property of holomorphic maps, we can prove:
    \begin{proposition}
    \label{prop:journeholomo}
    A transversely holomorphic foliation $\mathcal{F}$ with holomorphic leaves on a complex manifold $M$ is a holomorphic foliation.
    \end{proposition}
    \begin{proof}
        The idea of proof is similar to that of the Proposition \ref{prop:journe}. The difference is that we use a specific holomorphic chart around every point instead of the exponential map in \cite{pugh_holder_1997}.
        Denote by $k$ the complex dimension of $\mathcal{F}$ and $n$ the complex dimension of $M$.
        Fix a point $p \in M$ and a holomorphic chart $(U,\psi)$ centered at $p$, with $\psi:U\to V_1\times V_2$, where $V_1$ is the unit ball of $\mathbb C^k$ and $V_2$ is the unit ball of $\mathbb C^{n-k}$. The chart $(U,\psi)$ can be chosen so that $T\psi_*(\mathcal{F})$ is close to $\mathbb C^k \times \{0\}$.
        The leaves of $\mathcal{F}|_U$ project to horizontal graphs of holomorphic maps $g_y:V_1 \to V_2$, where $y \in V_2$, such that $g_y(0)=y$.\\
        The map $\phi:(x,y)\in V_1\times V_2\mapsto \psi^{-1} (g_y(x))\in U$ is a homeomorphism. 
        It is holomorphic with respect to $x$ since $\mathcal{F}$ has holomorphic leaves, and holomorphic with respect to $y$ since the holonomy of $\mathcal{F}$ is holomorphic (see remark ???).
        Therefore, by Osgood's Lemma, the map $\phi^{-1}$ is a holomorphic chart for $M$, which is moreover adapted to $\mathcal{F}$ by construction. This proves that $\mathcal{F}$ is a holomorphic foliation of $M$.
    \end{proof}

\subsection{Partially hyperbolic flows}
See \cite{hirsch_invariant_1977}, \cite{brin_partially_1974}, \cite{barreira_nonuniform_2007} for more details.
\begin{definition}
\label{def:parthypflo}
    Let $M$ a smooth compact manifold and $(\varphi^t)_{t \in \mathbb R}$ a smooth flow on $M$  generated by a nowhere vanishing smooth vector field $X$.\\
   The flow $(\varphi^t)$ is said to be \textit{(uniformly) partially hyperbolic} on $M$ if there exist $(d\varphi^t)$-invariant continuous subbundles of $TM$ - $E^{s}$, $E^{\hat{c}}$ and $E^{u}$ - a Riemannian metric $\|.\|$ on $M$, real numbers $\alpha <\alpha'< 1 < \beta'<\beta$ and $C>0$ such that:
    \begin{enumerate}
        \item $TM = E^{s} \oplus E^{\hat{c}} \oplus  \mathbb R X \oplus E^{u}$ ;
        \item For every $t\geq 0$:
        \label{eq:ano}
            \begin{alignat*}{2}
                &\forall v \in E^{s}, \quad &&\|d\varphi^t(v)\| \leq C\alpha^t\|v\|,\\
                &\forall v \in E^{u}, \quad &C^{-1}\beta^{t}\|v\| \leq{} &\|d\varphi^{t}(v)\|,\\
                &\forall v \in E^{\hat{c}}, \quad &C^{-1}\alpha'^{t}\|v\| \leq{} &\|d\varphi^{t}(v)\| \leq C\beta'^{t}\|v\| .
            \end{alignat*}
    \end{enumerate}
    The vector bundles $E^{s}$, $E^{\hat{c}}$ and $E^{u}$ are called respectively the \emph{(strong) stable, subcenter, \emph{and} (strong) unstable distributions} of $(\varphi^t)$, while $\mathbb R X \oplus E^{s}$, $E^c:=\mathbb R X \oplus E^{\hat{c}}$ and $\mathbb R X \oplus E^{u}$ are called respectively the \emph{weak stable, center \emph{and} weak unstable distributions} of $(\varphi^t)$.
\end{definition}
Remark that for every $x \in M$ and $t\geq 0$, $d_x\varphi^t(X(x))=X(\varphi^t(x))$, so if we note $C_1,C_2>0$ such that for every $x \in M$, $C_1 \leq \|X(x)\|\leq C_2$ (which exist since $M$ is compact and $X$ is continuous nowhere vanishing), it comes that:
$$\forall v \in \mathbb RX(x), \quad \frac{C_1}{C_2}\|v\| \leq{} \|d\varphi^{t}(v)\| \leq\frac{C_2}{C_1}\|v\| .$$
In particular, the time-$t$ map of a partially hyperbolic flow is a partially hyperbolic diffeomorphism whose stable, center and unstable distributions are $E^s, E^c$ and $E^u$ respectively.\\
This definition recovers that of an \textit{Anosov flow}, that is when the subcenter bundle $E^{\hat{c}}$ is trivial.\\
Via a change of Riemannian metric we can assume that $C = 1$ (see \cite{barreira_nonuniform_2007}).

Our definition of partially hyperbolic flow is identical to that of \cite{wang_quasi-shadowing_2023}, or \cite{carneiro_partially_2014}, and has the advantage of considering explicitly a subbundle of $TM$ which is transverse to the vector field inducing the flow. 
Therefore, one of the main techniques used in order to study a flow - which is assuming a transverse structure invariant by flow holonomies - can be very efficient.
For example, we will manage to prove that if the subcenter distribution of a transversely holomorphic partially hyperbolic flow is integrable to a flow-invariant foliation $\mathcal{F}^{\hat{c}}$, then its leaves are holomorphic and the flow acts holomorphically on it.
Also, it is sometimes difficult to extrapolate some techniques used in the discrete case: in \cite{bohnet_partially_2011} and \cite{bohnet_partially_2014}, the authors study partially hyperbolic diffeomorphisms whose center distribution is integrable (see the next subsection) to a compact foliation with finite holonomy.
In our case, such a phenomenon is impossible as the center foliation would be subfoliated by the orbit foliation which possesses non-compact leaves (the non periodic orbits).

\subsection{Dynamical coherence}
As was discussed in \cite{wilkinson_dynamical_2008}, one can define different notions of integrability of a continuous subbundle of $TM$.
We recall some of them.
Let $E$ a continuous subbundle of dimension $k$ of the tangent bundle of a smooth manifold $M$.
$E$ is said to be \textit{integrable} if there exists a foliation $\mathcal{F}$ of $M$ by $C^1$ immersed submanifolds of dimension $k$ which are everywhere tangent to $E$.
$E$ is said to be \textit{uniquely integrable} if it is integrable to a foliation $\mathcal{F}$ and every $C^1$ path of $M$ everywhere tangent to $E$ lies in a single leaf of $\mathcal{F}$.
The notion of unique integrability is different from the fact that there exists a unique foliation tangent to $E$.
If $E$ is uniquely integrable, then there exists a unique foliation whose leaves are everywhere tangent to $E$.

The stable and unstable bundles of a partially hyperbolic flow are uniquely integrable to invariant foliations $\mathcal{F}^s$ and $\mathcal{F}^u$ respectively, whose leaves are as smooth as the flow (see \cite{hirsch_invariant_1977} or \cite{barreira_nonuniform_2007}).

\begin{definition}
    We say that a partially hyperbolic flow $(\varphi^t)$ is \textit{dynamically coherent} if there exist flow-invariant foliations $\mathcal{F}^{\hat{c}s}$ and $\mathcal{F}^{\hat{c}u}$ whose leaves are everywhere tangent to $E^{\hat{c}s}=E^{s}\oplus E^{\hat{c}}$ and $E^{\hat{c}u}=E^{u}\oplus E^{\hat{c}}$ respectively.
\end{definition}

This implies in particular, by Lemma \ref{lem:weakfol}, that the time-one map $\varphi^1$ is dynamically coherent.

If a partially hyperbolic flow is dynamically coherent, then its subcenter distribution is integrable (Proposition 2.10 in \cite{abouanass_dynamical_2026}):

\begin{proposition}
\label{prop:dyncohsubfol}
    Let $(\varphi^t)$ a partially hyperbolic flow on a smooth compact manifold which is dynamically coherent.\\
    Then the subcenter distribution is integrable to a flow-invariant foliation $\mathcal{F}^{\hat{c}}$. 
    Moreover, $\mathcal{F}^s$ and $\mathcal{F}^{\hat{c}}$ subfoliate $\mathcal{F}^{\hat{c}s}$ while $\mathcal{F}^u$ and $\mathcal{F}^{\hat{c}}$ subfoliate $\mathcal{F}^{\hat{c}u}$.
\end{proposition}
\begin{definition}
For a partially hyperbolic flow $(\varphi^t)$ whose subcenter distribution is integrable to a foliation $\mathcal{F}^{\hat{c}}$, we say that its subcenter foliation is \textit{complete} if for every $x \in M$ and $* \in \{s,u\}$, $\bigcup_{z \in \mathcal{F}^{\hat{c}}(x)} \mathcal{F}^{*}(z)=\bigcup_{w \in \mathcal{F}^{*}(x)} \mathcal{F}^{\hat{c}}(w)$.
\end{definition}
If a partially hyperbolic flow admits a complete subcenter foliation, then it is dynamically coherent (see Proposition 6.4 of \cite{abouanass_dynamical_2026}).
Also :

\begin{theorem}
\label{thm:Fçcomp}
    Let $(\varphi^t)$ a partially hyperbolic flow on a smooth compact manifold $M$ whose subcenter distribution is integrable to a flow invariant compact foliation $\mathcal{F}^{\hat{c}}$ with trivial holonomy.\\
    Then $\mathcal{F}^{\hat{c}}$ is complete. 
    In particular, $(\varphi^t)$ is dynamically coherent.
\end{theorem}

\subsection{Global $su\Phi$-holonomy maps}

If $(\varphi^t)$ is a dynamically coherent partially hyperbolic flow, then each leaf of $\mathcal{F}^{\hat{c}s}$ is simultaneously subfoliated by $\mathcal{F}^{\hat{c}}$ and $\mathcal{F}^{s}$.
In particular, for any $x\in M$, we can define the holonomy of the continuous foliation $\mathcal{F}^s|_{\mathcal{F}^{\hat{c}s}(x)}$ of $\mathcal{F}^{\hat{c}s}(x)$, which we call the local stable holonomy. 
For $y \in \mathcal{F}^s(x)$, since stable leaves are contractible, the local stable holonomy of a path $\gamma$ in $\mathcal{F}^s$ between $x$ and $y$ with respect to $\mathcal{F}_{\text{loc}}^{\hat{c}}(x)$ and $\mathcal{F}_{\text{loc}}^{\hat{c}}(y)$ does not depend on $\gamma$.

\begin{definition}
    A dynamically coherent partially hyperbolic flow $(\varphi^t)$ admits \emph{global stable holonomy maps} if for any $x\in M$, $y\in \mathcal{F}^{s}(x)$, there exists a globally defined homeomorphism $h_{xy}^{s}: \mathcal{F}^{ \hat{c}}(x)\to \mathcal{F}^{\hat{c}}(y)$ such that for every $z\in \mathcal{F}^{\hat{c}}(x)$, $h_{xy}^{s}(z)\in \mathcal{F}^{s}(z)\cap \mathcal{F}^{\hat{c}}(y)$.\\
    Similarly, we can define global unstable holonomy maps $h^{u}$ and global flow holonomy maps $h^\Phi$.\\
    We say that $(\varphi^t)$ admits \emph{global $su\Phi$-holonomy maps} if it admits global stable, global unstable and global flow holonomy maps.
\end{definition}
Since global holonomy maps coincide locally with local holonomy maps, we use $h_{xy}^{s}$ to denote both local holonomy maps and global holonomy maps. 
If $\mathcal{F}^{\hat{c}}$ is invariant by the flow, then the latter admits global flow maps as we can define, for $x \in M$ and $y =\varphi^{t_0}(x) \in \Phi(x)$, the homeomorphism $z \in \mathcal{F}^{\hat{c}}(x) \mapsto \varphi^{t_0}(z) \in \mathcal{F}^{\hat{c}}(y)$ which coincides locally with local flow holonomy maps.

In that respect, by analogy with the discrete case (Lemmas 2.13 and 2.14 and of \cite{abouanass_dynamical_2026}):

\begin{lemma}
\label{lem:subcenthomeo}
    If $(\varphi^t)$ admits global $su\Phi$-holonomy maps, then all subcenter leaves are homeomorphic. 
\end{lemma}

\begin{lemma}
\label{lem:globalholo}
    Let $(\varphi^t)$ a partially hyperbolic flow on a smooth compact manifold $M$, with a flow invariant compact subcenter foliation $\mathcal{F}^{\hat{c}}$ with $C^1$-leaves and trivial holonomy.\\
    Then $(\varphi^t)$ admits global $su\Phi$-holonomies.
\end{lemma}

\subsection{Subcenter bunching}
\begin{definition}
    Let $(\varphi^t)$ a smooth partially hyperbolic flow on a smooth compact manifold $M$ such that there exists a flow-invariant subcenter foliation $\mathcal{F}^{\hat{c}}$ with $C^1$-leaves.\\
    We say that $(\phi^t)$ is \textit{subcenter bunched} if there exists $t_0>0$ such that for every $p \in M$,
    \begin{align*}
        \left\|\restr{d_p\varphi^{t_0}}{E^{s}(p)} \right\| \cdot&\left\|\restr{d_p\varphi^{t_0}}{E^{\hat{c}}(p)} \right\| < m\left( \restr{d_p\varphi^{t_0}}{E^{\hat{c}}(p)} \right)\\
        \left\|\restr{d_p\varphi^{t_0}}{E^{\hat{c}}(p)} \right\|<&\;m\left( \restr{d_p\varphi^{t_0}}{E^{\hat{c}}(p)} \right)\cdot m\left( \restr{d_p\varphi^{t_0}}{E^{u}(p)} \right).
    \end{align*}
\end{definition}

Again, by analogy with the discrete case (see \cite{pugh_holder_1997}) (Proposition 2.16 of \cite{abouanass_dynamical_2026}):

\begin{proposition}
\label{prop:subcenterbunchC1}
     Let $(\varphi^t)$ a smooth partially hyperbolic dynamically coherent flow which is subcenter bunched.\\
     Then for every $p \in M$, the restriction $\restr{\mathcal{F}^u}{\mathcal{F}^{\hat{c}u}(p)}$ to $\mathcal{F}^{\hat{c}u}(p)$ of the strong unstable foliation (respectively the restriction $\restr{\mathcal{F}^s}{\mathcal{F}^{\hat{c}s}(p)}$ to $\mathcal{F}^{\hat{c}s}(p)$ of the strong stable foliation) is $C^1$.
\end{proposition}

\subsection{Disintegration and Gibbs measures}

We recall the notion of disintegration of measures (see \cite{avila_absolute_2015}, \cite{butler_uniformly_2018} and \cite{avila_absolute_2022} for more details). Let $Z$ be a Polish space with a Borel probability measure $\mu$, and $\mathcal{P}$ a partition of $Z$ into measurable subsets.
Denote by $\overline{\mu}$ the induced measure on the $\sigma$-algebra generated by $\mathcal{P}$.

A \emph{disintegration} (or a \emph{system of conditional measures}) of $\mu$ with respect to the partition $\mathcal{P}$ is a family of probability measures $\{\mu_P\}_{P \in \mathcal{P}}$ on $Z$ with the following properties:
\begin{itemize}
    \item $P$ is of full $\mu_P$-measure for $\overline{\mu}$-almost every $P \in \mathcal{P}$;
    
    \item For any continuous function $f: Z \to \mathbb{R}$, the function $P \mapsto \int f \, d\mu_P$ is measurable and 
    \[
    \int_Z f \, d\mu = \int_{\mathcal{P}} \left( \int_P f \, d\mu_P \right) d\overline{\mu}(P).
    \]
\end{itemize}

The disintegration may not exist for an arbitrary given $\mathcal{P}$, but it always exists if $\mathcal{P}$ is a \emph{measurable partition} (see the above references for a definition):

\begin{theorem}[\cite{rokhlin_fundamental_nodate}]
\label{thm:rok}
With the previous notations, if $\mathcal{P}$ is a measurable partition, then there exists a system of conditional measures relative to $\mathcal{P}$. It is essentially unique in the sense that two such systems coincide in a set of full $\overline{\mu}$-measure.
\end{theorem}

The disintegration theorem of Rokhlin does not apply directly when a foliation has a positive measure set of non-compact leaves. On the other hand:

\begin{proposition}
Let $\mathcal{F}$ be a foliation on a manifold $M$, and let $\mu$ be a Borel probability measure on $M$. Assume that for $\mu$-almost every $x \in M$, the leaf $\mathcal{F}_x$ is compact.\\
Then $\mathcal{F}$ defines a measurable partition.\\
In particular, there exists a system of conditional measures relative to $\{\mathcal{F}_x: x \in M\}$ that is essentially unique.
\end{proposition}

If there exists a set of non-compact leaves of positive measure, then one must consider disintegrations into \emph{measures defined up to scaling}, that is, equivalence classes where one identifies any two (possibly infinite) measures that differ only by a constant factor. 
Instead of a Borel probability measure $\mu$, we consider a more general locally finite Borel measure $m$ on $M$.
Let $U$ be a foliation chart. By Rokhlin, there is a disintegration $\{m_{x}^{U}: x \in U\}$ of the restriction of $m$ to the foliation chart into conditional probabilities along the plaques of $\mathcal{F}$ in $U$, and this disintegration is essentially unique. The main point is that conditional measures corresponding to different foliation charts coincide on the intersection, up to a constant factor (see Lemma 3.2 of \cite{avila_absolute_2015}):

\begin{lemma}
For any foliation charts $U$ and $U'$ and for $m$-almost every $x \in U\cap U'$, the restrictions of $m_{x}^{U}$ and $m_{x}^{U'}$ to $U\cap U'$ coincide up to a constant factor.
\end{lemma}

Therefore, as has been discussed in \cite{avila_absolute_2022}, there exists a family $\{m^{\mathcal{F}}_{x}: x \in M\}$ where each $m^{\mathcal{F}}_{x}$ is a measure defined up to scaling with $m^{\mathcal{F}}_{x}(M \setminus \mathcal{F}_{x}) = 0$. 
The function $x \mapsto m^{\mathcal{F}}_{x}$ is constant on the leaves of $\mathcal{F}$, and the conditional probabilities $m_{x}^{U}$ along the plaques of any foliation chart $U$ coincide almost everywhere with the normalized restrictions of the $m^{\mathcal{F}}_{x}$ to the plaques of $U$.
Such a family is essentially unique and is called the \emph{disintegration of $m$}. 
We refer to the $m^{\mathcal{F}}_{x}$ as \emph{conditional classes of $m$} along the leaves of $\mathcal{F}$.\\
If the foliation $\mathcal{F}$ is invariant by a homeomorphism $f:M\to M$ (i.e. for every $x\in M$, $f(\mathcal{F}(x))=\mathcal{F}(f(x))$), and if $m$ is $f$-invariant, then by essential uniqueness of the disintegration of $m$, it comes $f_*(m_{x})=m_{f(x)}$ for a.e. $x\in M$.

In case $m$ is the Riemannian volume of a smooth compact manifold $M$ and $\mathcal{F}$ is a foliation of $M$ by $C^1$-leaves, then we say that \emph{$m$ has Lebesgue disintegration along $\mathcal{F}$} if for $m$-a.e. $x\in M$, the conditional measure $m^{\mathcal{F}}_{x}$ on the leaf $\mathcal{F}(x)$ is equivalent to the induced Riemannian volume on $\mathcal{F}(x)$.\\

Now let's consider a partially hyperbolic diffeomorphism $f$ on a smooth compact manifold $M$ (for example, the time-one map of a partially hyperbolic flow), and let $\mu$ a $f$-invariant Borel probability measure on $M$.

\begin{definition}
    The measure $\mu$ is a \emph{Gibbs $u$-state} if its conditional measures along the local unstable leaves are equivalent to the leafwise Lebesgue measures.\\
    We can define a \emph{Gibbs $s$-state} in the same way.
\end{definition}
See \cite{bonatti_dynamics_2010} for more details about Gibbs $u$-states. 
\begin{proposition}
\label{prop:gibbs}
    Let $f$ a $C^2$ partially hyperbolic diffeomorphism. \\
    Then:
    \begin{itemize}
        \item There exists at least one Gibbs $u$-state ;
        \item The ergodic components of a Gibbs $u$-state are also Gibbs $u$-states ;
        \item The support of any Gibbs $u$-state consists of entire unstable leaves.
    \end{itemize}
\end{proposition}
In particular, if $(\varphi^t)$ is a partially hyperbolic flow on a smooth manifold $M$, then $\varphi^1$ is a $C^\infty$ partially hyperbolic diffeomorphism, so there exist Gibbs $u$-states for $\varphi^1$, which we call \emph{Gibbs $u$-state for the flow $(\varphi^t)$}.

In case of a dynamically coherent partially hyperbolic diffeomorphism $f$, Xu and Zhang defined in \cite{xu_holomorphic_2025} a \emph{Gibbs $cu$-state} similarly as an $f$-invariant Borel probability measure whose conditional measures along the local center-unstable leaves are equivalent to leafwise Lebesgue measures. We can define a \emph{Gibbs $cs$-state} in the same way.
Such measures do not exist generally a priori, but one can be constructed under some assumptions on $f$ (see \cite{xu_holomorphic_2025}).

For a dynamically coherent partially hyperbolic flow $(\varphi^t)$, we define analogously:

\begin{definition}
    A $\varphi^1$-invariant Borel probability measure $\mu$ is a \emph{Gibbs $\hat{c}u$ state} if its conditional measures along the local subcenter-unstable leaves are equivalent to the leafwise Lebesgue measures.\\
    We can define a \emph{Gibbs $\hat{c}s$-state} similarly.
\end{definition}
We do not know wether such measures exist generally, but we'll prove the existence of one in section \ref{sec:5} which we'll allow us to prove the holomorphicity of the center-stable foliation of a fibered transversely holomorphic partially hyperbolic flow on a smooth compact seven-manifold.

\subsection{Quasiconformality}
We recall some notions of quasiconformality (see \cite{ahlfors_quasiconformal_1953} or \cite{vaisala_lectures_1971}).

\noindent Let $k\geq 2$ and $\|\cdot\|$ be the euclidean norm on $\mathbb R^k$.\\
\label{rem:quasiconfo}
A homeomorphism $h: U \to V$ between two open subsets $U, V$ of $\mathbb{R}^k$ is \emph{conformal} if it is $C^1$ and its differential at every point $x\in U$ is angle-preserving and orientation-preserving.
If $k=2$, then this definition coincides with that of a holomorphic function whose holomorphic derivative is nowhere vanishing.

For such maps $h:U\to V$ which are not $C^1$ we can define a more general notion of \emph{quasiconformality}.\\
We define the \emph{linear dilatation} of $h$ at $x \in U$ to be  
\[
L_h(x) = \limsup_{r \to 0} \frac{\displaystyle\max_{\|y-x\| = r} \|h(y) - h(x)\|}{\displaystyle\min_{\|y-x\| = r} \|h(y) - h(x)\|}.
\]  
If $L_h(x) \leq K$ for every $x \in U$ (where $K\geq1$), then $h$ is said to be \emph{$K$-quasiconformal}.
A homeomorphism $h:U \to V$ is \emph{quasiconformal} if there exists $K\geq1$ such that $h$ is $K$-quasiconformal.\\
We list some properties of quasiconformal maps:
\begin{itemize}
    \item If $h:U\to V$ is a diffeomorphism, then it is $K$-quasiconformal if and only if for every $x \in U$, 
$$\frac{\displaystyle\sup_{\|v\|=1} \|d_xh(v)\|}{\displaystyle\inf_{\|v\|=1} \|d_xh(v)\|}\leq K;$$
    \item If $h:U\to V$ is $K$ quasiconformal and $g:V \to W$ is $K'$ quasiconformal, then $g\circ h: U \to W$ is $KK'$ quasiconformal;
    \item A homeomorphism $h:U\to V$ is $1$-quasiconformal if and only if it is conformal;
    \item If a sequence $(h_n)$ of $K$-quasiconformal maps converges uniformly to a homeomorphism $h$, then $h$ is $K$-quasiconformal.
\end{itemize}
The main result concerning quasiconformal maps is the following:
\begin{lemma}
\label{lem:quasiconfholo}
    Let $h: U\to V$ a quasiconformal map between open subsets of $\mathbb R^k$.\\
    Then $h$ is absolutely continuous and differentiable Lebesgue almost everywhere.\\
    If $k=2$ and $\frac{\partial}{\partial\overline{z}}h(a)=0$ for Lebesgue $a.e.$ $a\in U$, then $h$ is conformal.
\end{lemma}

\section{Sketch of proof of the results}

In section \ref{sec:4}, we state some fundamental general results which will allow us to extrapolate some of the work done in the discrete case in ours. One of the results is Lemma \ref{lem:weakfol} which gives, for an invariant foliation $\mathcal{F}$ with $C^1$-leaves transverse to the direction of the flow, the existence of a foliation tangent to $T\mathcal{F}\oplus \mathbb R X$.
As was the case for a transversely holomorphic Anosov flow, we prove that the stable unstable and subcenter leaves of a transversely holomorphic partially hyperbolic flow have complex structures which are preserved by the action of the flow ( Proposition \ref{prop:leavesholo}). This is also true for the subcenter-unstable and subcenter-stable foliations if they exist
Thanks to the works of \cite{hirsch_invariant_1977}, the above implies that if moreover the subcenter distribution is of complex dimension one and the flow is dynamically coherent, then the unstable (respectively stable) foliation holomorphically subfoliates the subcenter-unstable (respectively subcenter-stable) leaves (Corollary \ref{cor:FçholoinFçs}).
We also use the works of \cite{kalinin_cocycles_2013} and \cite{sadovskaya_uniformly_nodate} to define the non-stationary linearization of $\mathcal{F}^u$ in case it is of complex dimension one (Proposition \ref{prop:non-statiolin}).
This leads to the two main results of this section which assert that the center stable foliation holonomy is quasiconformal if the subcenter foliation is compact with trivial holonomy (Proposition \ref{prop:hcsquasiconf}), and a version of Fubini's theorem on subcenter-unstable and subcenter-stable leaves in the seven-dimensional case (Proposition \ref{prop:fubini}).

From section \ref{sec:5} onward, we assume $(\varphi^t)$ is a smooth transversely holomorphic partially hyperbolic flow on a smooth compact connected seven-manifold whose subcenter distribution is integrable to a flow-invariant compact foliation with trivial holonomy. 
We rely on the results of \cite{abouanass_dynamical_2026} on dynamical coherence of partially hyperbolic flows.
We prove that the subcenter foliation holomorphically subfoliates the subcenter-unstable and subcenter-stable leaves. This is the analog result of the one proven in \cite{xu_holomorphic_2025} in the holomorphic discrete case. 
We follow the proofs of this paper to show in Proposition \ref{prop:isoorcont} that $(\varphi^t)$ is either of two types: a $\mu$-subcenter contraction or a  $\mu$-subcenter isometry.
As for the contraction case, we prove directly that the subcenter holonomy inside subcenter-unstable leaves is holomorphic for a.e. leaf (Proposition \ref{prop:contract1}).
As for the isometry case, we first need to prove the existence of a measure whose conditional measures along subcenter-unstable leaves is equivalent to the leafwise Lebesgue measures (Proposition \ref{prop:gibbsçu}) in order to use the same techniques mentioned in \cite{butler_uniformly_2018} to show that the subcenter holonomy can be extended, thanks to the holomorphic non-stationary linearization of $\mathcal{F}^u$, to the entire unstable leaves as a $\mathbb R$-linear map (Proposition \ref{prop:subcentholoRlin}) and eventually as a $\mathbb C$-linear map (Proposition \ref{prop:isometryholo}) for a.e. subcenter leaf.
We prove the main result (Proposition \ref{prop:concholo}) thanks to results on Gibbs $u$-states (see \cite{bonatti_dynamics_2010}) as well as the local product structure between the projection of the center-stable and center-unstable foliations on the quotient manifold $M/\mathcal{F}^{\hat{c}}$.

In section \ref{sec:6}, we use the previous results as well as some of \cite{butler_uniformly_2018} to prove that the center-stable, center-unstable and center foliations are transversely holomorphic and thus $C^\infty$ (Proposition \ref{prop:Fcstrholo}). This leads to the smoothness of the subcenter foliation (Corollary \ref{cor:Fçsmooth}).
By the results of \cite{abouanass_global_2025}, we conclude (Corollary \ref{cor:classif}).

\section{Some fundamental results}
\label{sec:4}

\subsection{Complex structures of the leaves}
The following result proven in \cite{abouanass_global_2025} guarantees, for every flow-invariant foliation $\mathcal{F}$ not containing the flow direction, the existence of a foliation $\mathcal{F}^w$ tangent to $T\mathcal{F}\oplus \mathbb RX$.
\begin{lemma}
\label{lem:weakfol}
     Let $(\varphi^t)_{t\in \mathbb R}$ a partially hyperbolic flow on a smooth manifold $M$. Let $*\in \{s,u,\hat{c}, \hat{c}s, \hat{c}u \}$. Assume the foliation $\mathcal{F}^*$ exists and is invariant by the flow.\\
     Then there exists a continuous foliation with $C^1$ leaves tangent to $E^* \oplus \mathbb R X$ whose leaf at $x \in M$ is $\mathcal{F}^{w*}_x:=\bigcup_{t \in \mathbb R}\varphi^t(\mathcal{F}^*_x)$.
\end{lemma}
We will simply note $\mathcal{F}^{w\hat{c}}$, $\mathcal{F}^{w\hat{c}s}$ and $\mathcal{F}^{w\hat{c}u}$ as $\mathcal{F}^{c}$, $\mathcal{F}^{cs}$ and $\mathcal{F}^{cu}$ respectively.

One can prove that the induced metric on the plaques of $\mathcal{F}^w$ is equivalent to the product of induced metrics on the plaques of $\mathcal{F}$ and $\Phi$ (Lemma 4.2 of \cite{abouanass_dynamical_2026}) :

\begin{lemma}
\label{lem:metricweak}
    Let $(\varphi^t)$ a smooth flow on a smooth compact manifold M, induced by a nowhere vanishing vector field $X$.
    Let $\mathcal{F}$ a continuous foliation of $M$ with $C^r$ leaves $(r \geq 1)$, which is invariant under the flow $\varphi^t$ and such that $X$ is nowhere tangent to $\mathcal{F}$. \\
    Then there exist positive constants $C_1<1$, $C_2>1$ and $\delta_0$ such that for every $\delta<\delta_0$ and $x \in M$, 
    $$\bigcup_{z \in  \Phi_{C_1\delta}(x)}\mathcal{F}_{C_1\delta}(z)\subset \mathcal{F}^{w}_{\delta}(x)\subset \bigcup_{z \in  \Phi_{C_2\delta}(x)}\mathcal{F}_{C_2\delta}(z)$$
    and
    $$\bigcup_{z \in \mathcal{F}_{C_1\delta}(x) }\Phi_{C_1\delta}(z)\subset \mathcal{F}^{w}_{\delta}(x)\subset \bigcup_{z \in \mathcal{F}_{C_2\delta}(x) }\Phi_{C_2\delta}(z).$$
\end{lemma}
\begin{lemma}
\label{lem:metricflo}
    Let $(\varphi^t)$ a smooth flow on a smooth compact manifold M, induced by a nowhere vanishing vector field $X$.\\
    Then there exist positive constants $C_1<1$, $C_2>1$ and $\delta_0$ such that for every $\delta<\delta_0$ and $x \in M$, 
    $$\bigcup_{t \in  \left]-C_1\delta, C_1\delta \right[}\varphi^t(x)\subset \Phi_\delta(x)\subset \bigcup_{t \in  \left]-C_2\delta, C_2\delta \right[}\varphi^t(x).$$
\end{lemma}

The next result proven in \cite{abouanass_global_2025} deals with flow holonomies with respect to transversals containing small open subsets of flow-invariant $C^0$ foliations transverse to the flow direction:
\begin{proposition}
\label{prop:holoflo}
    Let $(\varphi^t)_{t\in \mathbb R}$ a complete smooth flow on a smooth manifold M, induced by a nowhere vanishing vector field $X$.
    Let $\mathcal{F}$ a continuous foliation of $M$ with $C^1$ leaves, which is invariant under the flow $\varphi^t$ (i.e. for every $t \in \mathbb R$ and $x \in M$, $\varphi^t(\mathcal{F}_x)=\mathcal{F}_{\varphi^t(x)}$), such that $X$ is nowhere tangent to $\mathcal{F}$.\\
    For $x \in M$ and $t\in \mathbb{R}$, let $W_{\varphi^t(x)}$ be a small neighborhood of $\varphi^t(x)$ in $\mathcal{F}_{\varphi^t(x)}$ and by $T_{\varphi^t(x)}$ a small transversal to the flow in $\varphi^t(x)$ containing $W_{\varphi^t(x)}$. \\
    Fix $x\in M, t\in \mathbb R$ and let the path 
    $\beta^x_t: \left \{\begin{array}{ccc}
        [0,1] & \to & M \\
        s  &\mapsto  &\varphi^{st}(x)
    \end{array} \right. .$\\
    Then if $T_x$ is taken sufficiently small, the restriction to $W_x$ of the $\Phi$-holonomy of the path $\beta^x_t$ with respect to the transversals $T_{x}$ and $T_{\varphi^t(x)}$ is the restriction to $W_x$ of the map 
        \[\restr{\varphi^t}{\mathcal{F}_x}: \mathcal{F}_x \to \mathcal{F}_{\varphi^t(x)}.\]
\end{proposition}

\begin{corollary}
\label{cor:diffholo}
    Let $(\varphi^t)_{t\in \mathbb R}$ a partially hyperbolic flow on a smooth compact manifold $M$ with a flow invariant subcenter foliation $\mathcal{F}^{\hat{c}}$.
    For $*\in \{u,s,\hat{c}\}$, $x \in M$ and $t\in \mathbb{R}$, denote by $W^{*}_{\varphi^t(x)}$ a small neighborhood of $\varphi^t(x)$ in $\mathcal{F}^{*}_{\varphi^t(x)}$ and by $S_{\varphi^t(x)}$ a small transversal to the flow at $\varphi^t(x)$ containing $W^{u}_{\varphi^t(x)}$, $W^{s}_{\varphi^t(x)}$ and $W^{\hat{c}}_{\varphi^t(x)}$.\\
    Then $d_x(h(\beta^x_t)^{S_{\varphi^t(x)}, S_x})=d_x\varphi^t$.
    \end{corollary}
    \begin{proof}
        Since $T_xS_x=E^{u}_x \oplus E^{\hat{c}}_x \oplus E^{s}_x$, the previous corollary gives the result.
    \end{proof}

Therefore, we can prove the following essential result:
\begin{proposition}
\label{prop:leavesholo}
    Let $(\varphi^t)_{t\in \mathbb R}$ a transversely holomorphic partially hyperbolic flow on a smooth compact manifold $M$ with a flow invariant subcenter foliation $\mathcal{F}^{\hat{c}}$. Let $*\in \{s,u,\hat{c}\}$.\\
    Then each leaf of $\mathcal{F}^*$ is an immersed complex manifold and for any $t\in \mathbb{R}$, $x\in M$, the well defined map:
    \[\restr{\varphi^t}{\mathcal{F}^{*}_x}: \mathcal{F}^{*}_x \to \mathcal{F}^{*}_{\varphi^t(x)}\]
    is holomorphic with respect to these complex structures.\\
    Moreover, if $*\in \{s,u\}$ and $E^{\hat{c}*}$ is integrable to a flow-invariant foliation $\mathcal{F}^{\hat{c}*}$, then each leaf of $\mathcal{F}^{\hat{c}*}$ is an immersed complex manifold and for any $t\in \mathbb{R}$, $x\in M$, the well defined map:
    \[\restr{\varphi^t}{\mathcal{F}^{\hat{c}*}_x}: \mathcal{F}^{\hat{c}*}_x \to \mathcal{F}^{\hat{c}*}_{\varphi^t(x)}\]
    is holomorphic with respect to these complex structures.
\end{proposition}
\begin{proof}
    Fix $x \in M$. For any $t\in \mathbb R$, let $S_{\varphi^t(x)}$ a small transversal to the flow containing a small open neighborhood $W^{u}_{\varphi^t(x)}$ of $\varphi^t(x)$ in $\mathcal{F}^{u}_{\varphi^t(x)}$, a small open neighborhood $W^{s}_{\varphi^t(x)}$ of $\varphi^t(x)$ in $\mathcal{F}^{s}_{\varphi^t(x)}$ and a small open neighborhood $W^{\hat{c}}_{\varphi^t(x)}$ of $\varphi^t(x)$ in $\mathcal{F}^{\hat{c}}_{\varphi^t(x)}$.
    Let $\Phi$ the foliation defined by the transversely holomorphic flow $(\varphi^t)_t$. The normal bundle $\nu = TM\diagup \Phi$ is naturally isomorphic (in the $C^0$ category) to $E^{u} \oplus E^{\hat{c}} \oplus E^{s}$.
    Therefore, the smooth almost complex structure $I$ on $\nu$ invariant by holonomy maps (see Proposition \ref{prop:caractrholo}) gives rise to a continuous almost complex structure, still denoted by $I$, on $E^{u} \oplus E^{\hat{c}} \oplus E^{s}$ which makes it a continuous vector subbundle of $TM$ with complex fibers.
    Moreover, since $I$ is invariant under $\Phi$-holonomy maps, then the restriction $I^{(t)}$ of $I$ on $TS_{\varphi^t(x)}$ is a smooth almost complex structure on $S_{\varphi^t(x)}$ satisfying 
    \[I^{(t)}\circ d(h(\beta^x_t)^{S_{\varphi^t(x)}, S_x}) = d(h(\beta^x_t)^{S_{\varphi^t(x)}, S_x}) \circ I^{(0)}  \quad (*)\]
    Since each $I^{(t)}$ is integrable by Proposition \ref{prop:caractrholo}, we can restrict each $S_{\varphi^t(x)}$ so that it is diffeomorphic to an open set of $\mathbb C^p$ and $I^{(t)}$ coincides with the multiplication by $i \in \mathbb C$. 
    With that in mind, thanks to Corollary \ref{cor:diffholo}, by taking $S_{x}$ even smaller, it comes for $* \in \{u,s,\hat{c}\}$,  $v\in E^{*}_x$ and $t \in \mathbb R$, that 
    $id_x\varphi^t(v)=d_x\varphi^t(iv).$\\
    Take a continuous metric  $\|\cdot\|_{(0)}$ on $M$ which is hermitian on $E^{u} \oplus E^{\hat{c}} \oplus E^{s}$ with respect to $I$.
    It comes, for $v \in E^{u}_x$ and $t \in \mathbb R$:
    \begin{align*} \|d_x\varphi^t(iv)\|_{(0)}=\|id_x\varphi^t(v)\|_{(0)}=\|d_x\varphi^t(v)\|_{(0)} \leq C\lambda^t\|v\|_{(0)}= C\lambda^t\|iv\|_{(0)}.
    \end{align*}
    Now take a smooth Riemannian metric $\|\cdot\|$ on $M$.
    Since $M$ is compact and both metrics are continuous, $\|\cdot\|$ and $\|\cdot\|_{(0)}$ are equivalent on $M$.
    Therefore, by the above and the characterization of the strong unstable space, it comes $iv\in E^{u}_x$.
    The same can be said for the stable distribution.
    As for the subcenter distribution, for $v \in E^{\hat{c}}_x$ and $t \in \mathbb R$:
    \begin{align*} \|d_x\varphi^t(iv)\|_{(0)}=\|id_x\varphi^t(v)\|_{(0)}=\|d_x\varphi^t(v)\|_{(0)} &\leq C\beta'^t\|v\|_{(0)}= C\beta'^t\|iv\|_{(0)}\\
    &\geq C^{-1}\alpha'^t\|v\|_{(0)}=C^{-1}\alpha'^t\|iv\|_{(0)}.
    \end{align*} 
    By using a smooth Riemannian metric as above and the characterization of the center space, it comes $iv\in E^{c}_x=\mathbb R X(x) \oplus E^{\hat{c}}_x$. Since $I$ takes value in $E^{s} \oplus E^{\hat{c}} \oplus E^{u}$, necessarily $iv\in E^{\hat{c}}_x$.\\
    To summarize, for $* \in \{u,s,\hat{c}\}$, we have defined on $\mathcal{F}^{*}_x$ a continuous almost complex structure $J_*^{(x)}$ whose restriction to $W^{*}_y$ for every $y \in \mathcal{F}^{*}_x$ is a smooth almost complex structure, making $J_*^{(x)}$ a smooth almost complex structure on $\mathcal{F}^{*}_x$.\\
    Now fix $y \in \mathcal{F}^{*}_x$ and let $S_{y}$ a small transversal to the flow containing a small open neighborhood $W^{*}_{y}$ of $y$ in $\mathcal{F}^{*}_{x}$. Denote by $I_*^{(y)}$ the restriction of $I$ to $S_y$. We know it is integrable by Corollary \ref{prop:caractrholo}. Therefore, its Nijenhuis tensor $N_{I_*^{(y)}}$ vanishes everywhere. Since the restriction of $I_*^{(y)}$ to $W^{*}_y$ is exactly the restriction of $J_*^{(x)}$ to $W^{*}_y$, it follows that $N_{J_*^{(x)}}(X,Y)=0$ for any vector fields $X,Y$ tangent to $W^{*}_{y}$. Since $y \in \mathcal{F}^{*}_x$ is arbitrary, it comes $N_{J_*^{(x)}}=0$ so that $J_*^{(x)}$ is an integrable almost complex structure on $\mathcal{F}^{*}_x$. Moreover, the equality $(*)$ still holds for $y \in \mathcal{F}^{*}_x$ so that for any $t\in \mathbb R$: 
    \[J_*^{(\varphi^t(x))}\circ d(\restr{\varphi^t}{\mathcal{F}^{*}_x}) = d(\restr{\varphi^t}{\mathcal{F}^{*}_x}) \circ J_*^{(x)}.\]
    Since for any $t\in \mathbb R$, $J_*^{(\varphi^t(x))}$ is integrable, it comes that $d(\restr{\varphi^t}{\mathcal{F}^{*}_x})$ is $\mathbb C$-linear, which means exactly that
    \[\restr{\varphi^t}{\mathcal{F}^{*}_x}: \mathcal{F}^{*}_x \to \mathcal{F}^{*}_{\varphi^t(x)}\]
    is holomorphic with respect to the complex structures induced by $J_*^{(x)}$ and $J_*^{(\varphi^t(x))}$.\\
    If $\mathcal{F}^{\hat{c}*}$ exists for $* \in \{s,u\}$, then the above shows that $E^{\hat{c}*}$ is stable by $I$ so the latter defines, for $x \in M$, a smooth almost complex structure $J^{(x)}_{\hat{c}*}$ on $\mathcal{F}^{\hat{c}*}(x)$. 
    We can proceed as above to prove that $J^{(x)}_{\hat{c}*}$ is integrable and that the flow acts holomorphically on the subcenter-* leaves.
\end{proof}
\begin{corollary}
\label{cor:FçholoinFçs}
    Let $(\varphi^t)$ a transversely holomorphic partially hyperbolic dynamically coherent flow on a smooth compact manifold $M$. Assume $\dim_{\mathbb C} E^{\hat{c}}=1$.\\
    Then for every $p \in M$, the holonomy of the foliation $\restr{\mathcal{F}^u}{\mathcal{F}^{\hat{c}u}(p)}$ in $\mathcal{F}^{\hat{c}u}(p)$ (respectively the foliation $\restr{\mathcal{F}^s}{\mathcal{F}^{\hat{c}s}(p)}$ in $\mathcal{F}^{\hat{c}s}(p)$) is holomorphic.
\end{corollary}
\begin{proof}
    As $(\varphi^t)$ is dynamically coherent, there exists a flow-invariant foliation with holomorphic leaves tangent to $E^{\hat{c}}$ by \ref{prop:dyncohsubfol} and \ref{prop:leavesholo}.
    If $\dim_{\mathbb C} E^{\hat{c}}=1$, then since for every $x\in M$,
    \[\restr{\varphi^1}{\mathcal{F}^{\hat{c}}_x}: \mathcal{F}^{\hat{c}}_x \to \mathcal{F}^{\hat{c}}_{\varphi^t(x)}\]
    is holomorphic thus conformal, it comes that 
    $$\frac{\left\|\restr{d_x\varphi^{1}}{E^{\hat{c}}(x)} \right\|}{m\left( \restr{d_x\varphi^{1}}{E^{\hat{c}}(x)} \right)}=1.$$
    Therefore $(\varphi^t)$ is subcenter-bunched, so Proposition \ref{prop:subcenterbunchC1} applies and shows that for every $x \in M$,  the holonomy of the foliation $\restr{\mathcal{F}^s}{\mathcal{F}^{\hat{c}s}(x)}$ in $\mathcal{F}^{\hat{c}s}(x)$ is uniformly $C^1$ (as well as the holonomy of the foliation $\restr{\mathcal{F}^u}{\mathcal{F}^{\hat{c}u}(x)}$ in $\mathcal{F}^{\hat{c}u}(x)$).
    We prove that it is holomorphic for the stable case. The unstable case is analogous by reversing time.
    \begin{claim}
        For every $x \in M$, $y \in \mathcal{F}^{s}(x)$  and $n \in \mathbb N$, if $V$ is a sufficiently small open neighborhood of $x$ in $\mathcal{F}^{\hat{c}}(x)$, then
        $$h_{xy}^{s}|_V=\varphi^{-n}\circ h_{\varphi^{n}(x),\varphi^{n}(y)}^{s}\circ \varphi^{n}|_V.$$
    \end{claim}
    \begin{proof}
        This is immediate since $\mathcal{F}^{\hat{c}}$ is invariant by $(\varphi^t)$.
    \end{proof}
    Let $x \in M$, $y \in \mathcal{F}^{s}(x)$, $z\in \mathcal{F}^{\hat{c}}_{\mathrm{loc}}(x)$ and $n \in \mathbb N$. We write $w=h_{xy}^{s}(z)$.
    We differentiate the previous equality so
    \[d_zh_{xy}^{s}=(d_w\varphi^{n}|_{E^{\hat{c}}})^{-1}\circ d_{\varphi^n(z)}h_{\varphi^{n}(x)\varphi^{n}(y)}^{s}\circ d\varphi_{z}^{n}|_{E^{\hat{c}}}.\]
Since $d_z\varphi^n|_{E^{\hat{c}}}$ is conformal and $d_{\varphi^n(z)}h_{\varphi^{n}(x)\varphi^{n}(y)}^{s}$ is close to the identity as $n$ goes to $+\infty$, it comes that $d_zh_{xy}^{s}$ is conformal thus holomorphic and so is $h_{xy}^{s}$.
\end{proof}

\begin{corollary}
    Let $(\varphi^t)$ a transversely holomorphic partially hyperbolic dynamically coherent flow on a smooth compact manifold $M$. Assume $\dim_{\mathbb C} E^{\hat{c}}=1$.\\
    Then for every $p \in M$, $\restr{\mathcal{F}^u}{\mathcal{F}^{\hat{c}u}(p)}$ is a holomorphic foliation in $\mathcal{F}^{\hat{c}u}(p)$ and $\restr{\mathcal{F}^s}{\mathcal{F}^{\hat{c}s}(p)}$ is a holomorphic foliation in $\mathcal{F}^{\hat{c}s}(p)$.
\end{corollary}
\begin{proof}
    This is an immediate consequence of the previous result, Proposition \ref{prop:leavesholo} and Proposition \ref{prop:journeholomo}.
\end{proof}

\begin{corollary}
\label{cor:Fçholoequiv}
    Let $(\varphi^t)$ a transversely holomorphic partially hyperbolic dynamically coherent flow on a smooth compact manifold $M$. Assume $\dim_{\mathbb C} E^{\hat{c}}=1$ and $(\varphi^t)$ admits global $su$ holonomy maps.\\
    Then for every $x, y \in M$, $\mathcal{F}^{\hat{c}}(x)$ and $\mathcal{F}^{\hat{c}}(y)$ are holomorphically equivalent.
\end{corollary}
\begin{proof}
    Since $\mathcal{F}^{\hat{c}}$ is invariant by the flow, it admits global flow holonomy maps.
    Therefore, it admits global $su\Phi$ holonomy maps and all subcenter leaves are homeomorphic thanks to Lemma \ref{lem:subcenthomeo}.
    As the proof of this lemma shows, any such homeomorphism is obtained as a finite composition of:
    \begin{itemize}
        \item global stable holonomies inside some stable subcenter leaves ;
        \item global unstable holonomies inside some unstable subcenter leaves ;
        \item global flow holonomies inside some center leaves.
    \end{itemize}
    By Corollary \ref{cor:FçholoinFçs}, the two first types are composed of holomorphic maps since they coincide locally with local holonomy maps.
    By Propositions \ref{prop:holoflo} and \ref{prop:leavesholo}, the third type is also composed of holomorphic maps, which concludes.
\end{proof}

Let $(\varphi^t)_{t\in \mathbb R}$ a transversely holomorphic partially hyperbolic flow on a smooth compact manifold $M$ with a flow invariant subcenter foliation $\mathcal{F}^{\hat{c}}$. 
Then by the local product structure of $\mathcal{F}^u$ and $\mathcal{F}^{cs}$ as well as Proposition \ref{prop:leavesholo}, it comes that $\mathcal{F}^u$ is a continuous foliation with holomorphic leaves, in the sense that there exists a $C^0$ foliated atlas for $\mathcal{F}^u$ whose restriction to unstable plaques is a holomorphic diffeomorphism.
Therefore, if we assume $\dim_{\mathbb C}E^u=1$, one can reiterate the proof of Theorem A of \cite{abouanass_global_2025} to show:
\begin{proposition}
    Let $(\varphi^t)_{t\in \mathbb R}$ a transversely holomorphic partially hyperbolic flow on a smooth compact manifold $M$ with a flow invariant subcenter foliation $\mathcal{F}^{\hat{c}}$. Assume $\dim_{\mathbb C}E^u=1$.\\
    Then there exists a unique family of complex affine structures on the unstable leaves $(\mathcal{F}^{u}(x))_{x\in M}$, holomorphically compatible with the initial complex structures introduced in Proposition \ref{prop:leavesholo} such that: for every $x \in M$ and every $t \in \mathbb R$, the map 
    \[\restr{\varphi^t}{\mathcal{F}^{u}(x)}: \mathcal{F}^{u}(x) \to \mathcal{F}^{u}(\varphi^t(x))\] 
    is affine.\\
Moreover, each of these complex affine structures is complete.
\end{proposition}
The previous result remains true if we replace "unstable" by "stable". 
\subsection{Non-stationary linearization and holonomy of $d\varphi^1|_{E^u}$}
Let $(\varphi^t)_{t\in \mathbb R}$ a transversely holomorphic partially hyperbolic flow on a smooth compact manifold $M$ with a flow invariant subcenter foliation $\mathcal{F}^{\hat{c}}$.
We proceed as in \cite{xu_holomorphic_2025}.
By \cite{kalinin_cocycles_2013}, the bundle $E^{u}$ is a Hölder continuous subbundle of $TM$ with some Hölder exponent $\beta>0$. Therefore $d\varphi^1|_{E^{u}}$ defines a Hölder continuous linear cocycle over $\varphi^1$ as been defined in \cite{kalinin_cocycles_2013}. 
For $x,y\in M$ two points closed enough, we define $I_{xy}:E^{u}_{x}\to E^{u}_{y}$ to be a linear identification which is $\beta$-Hölder close to the identity.
By Proposition \ref{prop:leavesholo}, $d\varphi^1|_{E^{u}}$ is holomorphic thus conformal if $\dim_{\mathbb C} E^u=1$ so it is fiber bunched in the sense of the authors. As a result, it follows:

\begin{proposition}
\label{prop:Hu}
For every $x\in M$ and $y\in \mathcal{F}^{u}(x)$, the quantity
\[
H^{u}_{xy}:=\lim_{n\to+\infty}d_{\varphi^{-n}y}\varphi^{n}|_{E^{u}}\circ I_{\varphi^{-n}x\varphi^{-n}y}\circ d_{x}\varphi^{-n}|_{E^{u}}
\]
exists and defines a $\mathbb{R}$-linear map from $E^{u}_{x}$ to $E^{u}_{y}$ with the following properties:
\begin{enumerate}[label=(\roman*)]
    \item $H^{u}_{xx}=\operatorname{Id}$ and $H^{u}_{yx}\circ H^{u}_{xy}=H^{u}_{xx}$;
    \item $H^{u}_{xy}=d_{\varphi^{-n}y}\varphi^{n}|_{E^{u}}\circ H^{u}_{\varphi^{-n}x\varphi^{-n}y}\circ d_{x}\varphi^{-n}$ for any $n\in \mathbb N$;
    \item $\|H^{u}_{xy}-I_{xy}\|\leqslant Cd(x,y)^{\beta}$, where $\beta$ is the exponent of Hölder continuity for $E^{u}$.
\end{enumerate}
\end{proposition}

Similarly, if $y \in \mathcal{F}^s(x)$, then the quantity
\[
\lim_{n\to+\infty}(D\varphi^{n}_{y}|_{E^{u}})^{-1}\circ I_{\varphi^{n}x\varphi^{n}y}\circ D\varphi^{n}_{x}|_{E^{u}}:=H^{s}_{xy}
\]
exists and gives a linear map from $E^{s}_{x}$ to $E^{s}_{y}$ with analogous properties. $H^{u}$ and $H^{s}$ are called the \emph{unstable holonomy} and \emph{stable holonomy} for $d\varphi^1|_{E^{u}}$ respectively.

\begin{remark}
Since $d\varphi^1|_{E^{u}}$ is $\mathbb{C}$-linear and $I_{xy}$ is close to a conformal linear map when $x$ and $y$ are close, $H^{u}$ and $H^{s}$ are actually $\mathbb{C}$-linear.
\end{remark}

We now introduce the non-stationary linearization of $\mathcal{F}^{u}$ as defined in \cite{sadovskaya_uniformly_nodate} for a discrete dynamical system:

\begin{proposition}   
\label{prop:non-statiolin}
Let $(\varphi^t)$ a transversely holomorphic partially hyperbolic flow on a smooth compact manifold $M$. Assume $\dim_{\mathbb C} E^u=1$.\\
Then for every $x \in M$, there exists a holomorphic diffeomorphism $\theta_x: E^u_x \to \mathcal{F}^u(x)$ such that:
\begin{enumerate}[label=(\roman*)]
    \item For every $t \in \mathbb R, \; \theta_{\varphi^t(x)}\circ d_x\varphi^t=\varphi^t\circ \theta_x$;
    \item $\theta_x(0)=x$ and $d_0\theta_x=\text{Id}$;
    \item The family of diffeomorphisms $(\theta_x)_{x \in M}$ varies continuously with $x$.
\end{enumerate}
Moreover, such a family is unique.\\
The same can be said with the strong stable distribution.
\end{proposition}
\begin{proof}
    Let $t_0<0$. Since $\dim_{\mathbb C} E^u=1$,  $(\varphi^{t_0},\mathcal{F}^u)$ satisfies the assumptions of Proposition 4.1 of \cite{sadovskaya_uniformly_nodate}. 
    Therefore, there exists a unique family of diffeomorphisms $(\theta^{(t_0)}_x)_{x \in M}$ satisfying:
    \begin{itemize}
        \item For every $x\in M$, $\theta^{(t_0)}_{\varphi^{t_0}(x)}\circ d_{x}\varphi^{t_0}=\varphi^{t_0}\circ \theta^{(t_0)}_{x}$.
        \item $\theta^{(t_0)}_x(0)=x$ and $d_0\theta^{(t_0)}_x=\text{Id}$;
        \item The family of diffeomorphisms $(\theta^{(t_0)}_x)_{x \in M}$ varies continuously with $x$.
    \end{itemize}
    \begin{lemma}
        For every $t\in \mathbb R^*_-$ and $x \in M$, $\theta^{(t)}_x=\theta^{(-1)}_x$.
    \end{lemma}
    \begin{proof}
        (see also \cite{fang_rigidity_2007}). \\
    We first prove that for every $k \in \mathbb N^*$ and $x \in M$, $\theta^{(-k)}_x=\theta^{(-1)}_x$, by induction on $k$.
    For $k=1$ it is immediate. Let $k \in \mathbb N^*$ fixed and assume that for every $x \in M$, $\theta^{(-k)}_x=\theta^{(-1)}_x$.
    We prove $\theta^{(-1)}_x=\theta^{(-k-1)}_x$ for every $x\in M$.
    In order to do that, we show that for every $x\in M$, $\theta^{(-1)}_{\varphi^{-k-1}(x)}\circ d_{x}\varphi^{-k-1}=\varphi^{-k-1}\circ \theta^{(-1)}_{x}$.
    Then the result will follow by uniqueness of $(\theta^{(-k-1)}_x)_{x \in M}$ as the second and third points are obviously true for $(\theta^{(-1)}_x)_{x \in M}$.
    By assumption, for every $x \in M$, $\theta^{(-1)}_{\varphi^{-k}(x)}\circ d_{x}\varphi^{-k}=\varphi^{-k}\circ \theta^{(-1)}_{x}$.
    By composing this equality on the left with $\varphi^{-1}$, it comes 
    $$\varphi^{-k-1}\circ \theta^{(-1)}_{x}=\varphi^{-1}\circ \theta^{(-1)}_{\varphi^{-k}(x)}\circ d_{x}\varphi^{-k}=\theta^{(-1)}_{\varphi^{-1}(\varphi^{-k}(x))}\circ d_{\varphi^{-k}(x)}\varphi^{-1}\circ d_{x}\varphi^{-k}=\theta^{(-1)}_{\varphi^{-k-1}(x)}\circ d_{x}\varphi^{-k-1}$$
    which concludes.\\
    Then we prove that for every $k,m \in \mathbb N^*$ and $x\in M$, $\theta^{(-k/m)}_x=\theta^{(-k)}_x$.
    Fix such $k \in \mathbb N^*$ and $m \in \mathbb N^*$ ($m\geq2$).
    For every $x \in M$, $\theta^{(-k/m)}_{\varphi^{-k/m}(x)}\circ d_{x}\varphi^{-k/m}=\varphi^{-k/m}\circ \theta^{(-k/m)}_{x}$.
    By composing this equality on the left with $\varphi^{-k/m}$, it comes 
    \begin{align*}
    \varphi^{-2k/m}\circ \theta^{(-k/m)}_{x}=\varphi^{-k/m}\circ \theta^{(-k/m)}_{\varphi^{-k/m}(x)}\circ d_{x}\varphi^{-k/m}&=\theta^{(-k/m)}_{\varphi^{-k/m}(\varphi^{-k/m}(x))}\circ d_{\varphi^{-k/m}(x)}\varphi^{-k/m} \circ d_{x}\varphi^{-k/m}\\&=\theta^{(-k/m)}_{\varphi^{-2k/m}(x)}\circ d_{x}\varphi^{-2k/m}.
    \end{align*}
    By induction, it comes that for $r \in \llbracket 1, m \rrbracket$ and $x \in M$, 
    $$\theta^{(-k/m)}_{\varphi^{-rk/m}(x)}\circ d_{x}\varphi^{-rk/m}=\varphi^{-rk/m}\circ \theta^{(-k/m)}_{x}.$$
    Letting $r=m$, this proves, by uniqueness of $(\theta^{(-k)}_x)_{x \in M}$, that for every $x \in M$, $\theta^{(-k/m)}_x=\theta^{(-k)}_x$.
    Therefore, for every $r \in \mathbb Q^*_-$, for every $x \in M$,
    $\theta^{(r)}_x=\theta^{(-1)}_x$.
    Now let $t \in \mathbb R^*_-$ and $(r_n)\in \mathbb Q^{\mathbb N}$ a sequence of rational numbers converging to $t_0$.
    Then for $n$ large enough, $r_n \in \mathbb Q^*_-$ so for every $x \in M$, $\theta^{(r_n)}_x=\theta^{(-1)}_x$.
    In particular, for every $x \in M$, $v \in E^u_{_x}$ and $n$ large enough,
    $$\theta^{(-1)}_{\varphi^{r_n}(x)} \left( d_{x}\varphi^{r_n}(v) \right)=\varphi^{r_n} \left( \theta^{(-1)}_{x}(v) \right).$$
    Since $\theta^{(-1)}_{y}$ varies continuously with $y \in M$, it comes, by letting $n \to +\infty$, that $\theta^{(-1)}_{\varphi^{t}(x)} \left( d_{x}\varphi^{t}(v) \right)=\varphi^{t} \left( \theta^{(-1)}_{x}(v) \right).$
    This proves the result by uniqueness of $(\theta^{(-1)}_{x})_{x \in M}$.
    \end{proof}
    Now let, for $x \in M$, $\theta_x:=\theta_x^{(-1)}$. 
    Then the family of diffeomorphisms $(\theta_x)_{x \in M}$ satisfies points $(ii)$ and $(iii)$.
    Also, since for every $t\in \mathbb R^*_-$ and $x \in M$, 
    $\theta^{(-1)}_{x}=\theta^{(t)}_{x}$, it comes $\theta_{\varphi^{t}(x)}\circ d_{x}\varphi^{t}=\varphi^{t}\circ \theta_{x}$ for $t\in \mathbb R^*_-$ and $x \in M$.
    Now let $t \in \mathbb R^*_+$ and $x \in M$.
    It comes by applying the previous equality with $-t \in \mathbb R^*_{-}$ and $\varphi^{t}(x)\in M$:
    $$\theta_{x}\circ d_{\varphi^{t}(x)}\varphi^{-t}=\theta_{\varphi^{-t}(\varphi^t(x))}\circ d_{\varphi^{t}(x)}\varphi^{-t}=\varphi^{-t}\circ \theta_{\varphi^t(x)}.$$
    By composing on the left with $\varphi^t$, and on the right with $d_{x}\varphi^t$, it comes $\theta_{\varphi^t(x)}\circ d_x\varphi^t=\varphi^t\circ \theta_x$ which is the result for $t\in \mathbb R^*_+$.
    As for $t=0$ this is immediate. 
    
    We eventually prove that for every $x \in M$, $\theta_x:E^u_x \to \mathcal{F}^u(x)$ is holomorphic.
    We use the following Lemma, which is Proposition 9 in \cite{butler_uniformly_2018} ($\varphi^1$ is a $C^\infty$ uniformly $u$-quasiconformal partially hyperbolic diffeomorphism):
    \begin{lemma}
    \label{lem:derivHu}
        For every $x\in M$ and $y\in \mathcal{F}^{u}(x)$, the map $\theta^{-1}_{y}\circ\theta_{x}:E^{u}_{x}\to E^{u}_{y}$ is an affine map with derivative $H^{u}_{xy}$.
    \end{lemma}
    As a result of the previous Lemma and property $(ii)$ of the Proposition, let $x \in M$, $v \in E^u_x$ and $y = \theta_x(v) \in \mathcal{F}^u(x)$.
    Then
    $$d_v\theta_x=d_{v}(\theta_y\circ (\theta_{y}^{-1}\circ\theta_{x}))=d_{0}\theta_y\circ d_v(\theta_{y}^{-1}\circ\theta_{x})=d_v(\theta_{y}^{-1}\circ\theta_{x})=H^u_{xy}$$
    is $\mathbb C$-linear, which exactly means that $\theta_x:E^u_x \to \mathcal{F}^u(x)$ is holomorphic.
\end{proof}
\subsection{Subcenter non-expansiveness and quasiconformality of the center-stable holonomy}
The main result of this subsection is Proposition \ref{prop:hcsquasiconf} which asserts that under the holomorphic coordinates given by the non-stationary linearization of $\mathcal{F}^u$, the center-stable holonomy is quasiconformal. 
Combined with some holomorphic properties, this will lead to the transverse holomorphic structure of the center-stable foliation and thus to the classification.

\begin{definition}
    Let $(\varphi^t)$ a partially hyperbolic smooth flow on a smooth manifold $M$. Assume $\mathcal{F}^*$ exists, where $ * \in \{s, u, \Phi, \hat{c}, \hat{c}s, \hat{c}u , c , cs, cu\} $.\\
    A path $\gamma: [0, 1] \to M$ is called a \emph{good (local) $*$-path} if $ \gamma $ is piecewise $ C^1 $ and lies entirely in one (local) $\mathcal{F}^*$-leaf.
    \end{definition}

\begin{definition}
    Let $(\varphi^t)$ a partially hyperbolic flow on a smooth manifold $M$, with a flow invariant subcenter foliation $\mathcal{F}^{\hat{c}}$ with $C^1$-leaves.\\
    $(\varphi^t)$ is called \emph{subcenter non-expansive} if there exists $ l > 0 $ which satisfies the following property: for any $ x \in M $, $ y \in \mathcal{F}^{\hat{c}}_{\operatorname{loc}}(x) $, $ n \in \mathbb N $ and any good local $ \hat{c} $-path $ \gamma $ from $ x $ to $ y $, there exists a good local $ \hat{c} $-path $ \gamma_n $ from $ \varphi^n(x) $ to $ \varphi^n(y) $ whose length is less than $l$ and such that $ \varphi^n(\gamma) \cdot \gamma_n^{-1} $ represents the identity element in $\text{Hol}(\mathcal{F}^{\hat{c}}(\varphi^n(x)), \varphi^n(x))$.
\end{definition}

\begin{lemma}
    Any partially hyperbolic flow on a smooth manifold $M$, with a flow invariant compact subcenter foliation $\mathcal{F}^{\hat{c}}$ with $C^1$-leaves and trivial holonomy is subcenter non-expansive.
\end{lemma}
\begin{proof}
    As $\mathcal{F}^{\hat{c}}$ is a uniformly compact foliation with trivial holonomy, Lemma 3.5 of \cite{bohnet_partially_2014} implies that there exists a uniform bound $l>0$ on the diameter of the subcenter leaves.
    Now let $x \in M$, $y \in \mathcal{F}^{\hat{c}}_{\text{loc}}(x)$ and $\gamma$ a good local $\hat{c}$-path from $x$ to $y$. 
    For $n \in \mathbb N$, let $\gamma_n$ a minimizing geodesic in $\mathcal{F}^{\hat{c}}(\varphi^n(x))$ between $\varphi^n(x)$ and $\varphi^n(y)$.
    Then for every $n \in \mathbb N$, $l(\gamma_n)<l$ and $\varphi^n(\gamma)\cdot\gamma_n^{-1}$ is a loop in $\mathcal{F}^{\hat{c}}(\varphi^n(x))$ so it represents the identity element in $\text{Hol}(\mathcal{F}^{\hat{c}}(\varphi^n(x)), \varphi^n(x))$ since $\mathcal{F}^{\hat{c}}$ has trivial holonomy.
    \end{proof}

From now on, when we consider the holonomy of any foliation between two points in the same local leaf, we will omit to precise the path that links them as Definition \ref{defprop:hol} remarks that if the two points are close enough, then the holonomy does not depend on that path. 

The proof of the following result is very similar to that of \cite{butler_uniformly_2018}. We just have to take care of some subtlety specific to our flow.

\begin{proposition}
\label{prop:hcsquasiconf}
    Let $(\varphi^t)$ a transversely holomorphic partially hyperbolic flow on a smooth compact manifold $M$, with a flow invariant compact subcenter foliation $\mathcal{F}^{\hat{c}}$ with $C^1$-leaves and trivial holonomy. Assume $\dim_{\mathbb C} E^{u}=1$ and let $(\theta_x)_{x \in M}$ the family of diffeomorphisms given by Proposition \ref{prop:non-statiolin}.\\
    Then there exists $K\geq1$ such that for every $x \in M$ and $y\in \mathcal{F}^{cs}_{loc}(x)$, the map
    $$\theta_y^{-1} \circ h^{cs}_{x,y}\circ \theta_x: \theta_x^{-1}(\mathcal{F}^u_{loc}(x)) \to \theta_y^{-1}(\mathcal{F}^u_{loc}(y))$$
    is $K$-quasiconformal.
\end{proposition}
\begin{proof}
    Remark in \cite{abouanass_dynamical_2026} that, by construction, $\mathcal{F}^{cs}$ coincides with $(\mathcal{F}^{\hat{c}s})^w$ with the notations of Lemma \ref{lem:weakfol}.
    Therefore, let $x \in M$ and $y\in \mathcal{F}^{cs}_{\operatorname{loc}}(x)$.
    By Lemma \ref{lem:metricweak}, there exist $y'\in \mathcal{F}^{\hat{c}s}_{\operatorname{loc}}(x)$ and a small $t \in \mathbb R$ such that $y=\varphi^t(y')$.
    Therefore, $h^{cs}_{x,y}=h^{cs}_{y',y} \circ h^{cs}_{x,y'}$, so
    $$\theta_y^{-1} \circ h^{cs}_{x,y}\circ \theta_x=(\theta_{y}^{-1} \circ h^{cs}_{y',y}\circ \theta_{y'}) \circ (\theta_{y'}^{-1} \circ h^{cs}_{x,y'}\circ \theta_x).$$
    The center-stable holonomy between $y'$ and $y \in \Phi_{\text{loc}}(y')$ with respect to $\mathcal{F}^u_{\text{loc}}(y')$ and $\mathcal{F}^u_{\text{loc}}(y)$ coincides with the restriction to $\mathcal{F}^u_{\text{loc}}(y')$ of the $\Phi$-holonomy between $y'$ and $y$ with respect to small transversals containing $\mathcal{F}^u_{\text{loc}}(y')$ and $\mathcal{F}^u_{\text{loc}}(y)$ respectively, since the strong unstable foliation is invariant by the flow.
    By Proposition \ref{prop:holoflo}, this means that $h^{cs}_{y',y}=\varphi^{t}|_{\mathcal{F}^u_{\text{loc}(y')}}:\mathcal{F}^u_{\text{loc}(y')}\to \mathcal{F}^u_{\text{loc}(y)}$ which is holomorphic by Proposition \ref{prop:leavesholo}.
    As a result, since for every $x' \in M$ the map $\theta_{x'}:E^u_{x'} \to \mathcal{F}^u(x')$ is holomorphic, the map
    $\theta_{y}^{-1} \circ h^{cs}_{y',y}\circ \theta_{y'}$ is holomorphic thus conformal because $\dim_{\mathbb C} E^{u}=1$. 
    Therefore, we can assume $y \in \mathcal{F}^{\hat{c}s}_{\text{loc}}(x)$.
    \begin{lemma}
        There exists $C_0>0$ such that for every $x\in M$ and $y \in \mathcal{F}^{\hat{c}s}_{\operatorname{loc}}(x)$, for every $n \in \mathbb N$, there exists a good $\hat{c}s$-path $\gamma_n$ between $\varphi^n(x)$ and $\varphi^n(y)$, with $l(\gamma_n)\leq C_0$,  satisfying:
        $$\theta_y^{-1}\circ h^{cs}_{x,y}\circ \theta_x=d_y\varphi^{-n}\circ \theta_{\varphi^n(y)}^{-1}\circ h^{cs}_{\gamma_n}\circ \theta_{\varphi^n(x)}\circ d_x\varphi^n.$$
    \end{lemma}
    \begin{proof}
        The proof is analogous to that of Lemma 13 in \cite{butler_uniformly_2018}: we replace $f$ with $\varphi^1$ and the word "center" by "subcenter" since our flow is subcenter non-expansive.
    \end{proof}
    Remark that since $\dim_{\mathbb C}E^u=1$, $\varphi^1$ is $u$-quasiconformal in the sense of \cite{butler_uniformly_2018}. Therefore we can proceed exactly as in their Lemma 12 by replacing $f$ with $\varphi^1$ and considering our coordinates charts $(\theta_x)_{x\in M}$ as well as the previous lemma.
\end{proof}
\subsection{Fubini's theorem}
We state a version of Fubini's theorem in our context which comes from the absolute continuity of the center-stable foliation (see Lemma \ref{lem:quasiconfholo} since the center-stable holonomy is quasiconformal by Proposition \ref{prop:hcsquasiconf}) as well as the local product structure of the (sub)center-unstable foliation given in Proposition 6.12 of \cite{abouanass_dynamical_2026}, since the subcenter foliation has trivial holonomy.
We use mainly the results of \cite{butler_uniformly_2018}.
\begin{proposition}
\label{prop:fubini}
    Let $(\varphi^t)$ a transversely holomorphic partially hyperbolic flow on a smooth compact manifold $M$ of dimension $7$, with a flow invariant compact subcenter foliation $\mathcal{F}^{\hat{c}}$ with $C^1$-leaves and trivial holonomy. \\
    Then for $* \in \{cu,cs,c,\hat{c}u,\hat{c}s, \hat{c}\}$, $m$ has Lebesgue disintegration along $\mathcal{F}^{*}$ leaves and for $*\in \{u,s\}, \;x\in M$:
    \[m_x^{\hat{c}*} \asymp \int_{\mathcal{F}^{\hat{c}}(x)} m^*_y \, dm^{\hat{c}}_{x}(y)\asymp \int_{\mathcal{F}^{*}_{\operatorname{loc}}(x)} m^{\hat{c}}_y \, dm^{*}_{x}(y)\]
    where $\asymp$ means that the two involved measures are equivalent.
\end{proposition}
\begin{proof}
    Since $M$ is of dimension $7$ and each stable, subcenter and unstable bundle is non trivial, it comes by Proposition \ref{prop:leavesholo} that for $* \in \{s,\hat{c},u\}$, $\dim_{\mathbb C} E^*=1$.
    By Proposition \ref{prop:hcsquasiconf}, Lemma \ref{lem:quasiconfholo} and Proposition 16 of \cite{butler_uniformly_2018}, since $\mathcal{F}^u$ is a strongly absolute continuous foliation (see \cite{barreira_nonuniform_2007}), it comes that $m$ has Lebesgue disintegration along $\mathcal{F}^{cs}$ leaves.
    In the same way, $m$ has Lebesgue disintegration along $\mathcal{F}^{cu}$ leaves.
    Let $x \in M$ fixed. 
    The center stable leaf $\mathcal{F}^{cs}(x)$ is foliated by the two transverse foliations $\Phi|_{\mathcal{F}^{cs}(x)}$ and $\mathcal{F}^{\hat{c}s}|_{\mathcal{F}^{cs}(x)}$. 
    As for $\Phi|_{\mathcal{F}^{cs}(x)}$ , it is $C^1$ thus strongly absolutely continuous. On the other hand, the holonomy of $\mathcal{F}^{\hat{c}s}|_{\mathcal{F}^{cs}(x)}$ between two nearby points $y\in \mathcal{F}^{cs}(x)$ and $z \in \mathcal{F}^{\hat{c}s}_{\text{loc}}(y)$ with respect to the transversals $\Phi_{\text{loc}}(y)$ and $\Phi_{\text{loc}}(z)$ is $C^1$ (thus absolutely continuous): it coincides with the map 
    $$p=\varphi^{t}(y)\in \Phi_{\text{loc}}(y) \mapsto \varphi^t(z)\in \Phi_{\text{loc}}(z).$$
    Therefore $\mathcal{F}^{\hat{c}s}|_{\mathcal{F}^{cs}(x)}$ is an absolutely continuous foliation of $\mathcal{F}^{cs}(x)$.
    By Proposition 16 of \cite{butler_uniformly_2018}, since $m$ has Lebesgue disintegration along $\mathcal{F}^{cs}$ leaves, it comes that $m_x^{cs}$ has Lebesgue disintegration along $\mathcal{F}^{\hat{c}s}|_{\mathcal{F}^{cs}(x)}$, so $m$ has Lebesgue disintegration along $\mathcal{F}^{\hat{c}s}$.
    The equivalence of measures is given by Corollary 17 of \cite{butler_uniformly_2018} and the local product structure in $\mathcal{F}^{\hat{c}u}(x)$ given in Prop 6.12 of \cite{abouanass_dynamical_2026}.\\
    Also, all of the above can be stated if we switch the stable and unstable roles.\\
    We now prove that $m$ has Lebesgue disintegration along $\mathcal{F}^c$ and $\mathcal{F}^{\hat{c}}$.
    The center unstable leaf $\mathcal{F}^{cu}(x)$ is foliated by the two transverse foliations $\mathcal{F}^u|_{\mathcal{F}^{cu}(x)}$ and $\mathcal{F}^{c}|_{\mathcal{F}^{cu}(x)}$. 
    As for $\mathcal{F}^u|_{\mathcal{F}^{cu}(x)}$ , it is strongly absolutely continuous. On the other hand, the holonomy of $\mathcal{F}^{c}|_{\mathcal{F}^{cu}(x)}$ between two nearby points $y\in \mathcal{F}^{cu}(x)$ and $z \in \mathcal{F}^{c}_{\text{loc}}(y)$ with respect to the transversals $\mathcal{F}^u_{\text{loc}}(y)$ and $\mathcal{F}^u_{\text{loc}}(z)$ coincides with the holonomy of $\mathcal{F}^{cs}$ between $y$ and $z$ which is quasiconformal by Proposition \ref{prop:hcsquasiconf}, thus absolutely continuous.
    Therefore $\mathcal{F}^{c}|_{\mathcal{F}^{cu}(x)}$ is an absolutely continuous foliation of $\mathcal{F}^{cu}(x)$.
    By Proposition 16 of \cite{butler_uniformly_2018}, since $m$ has Lebesgue disintegration along $\mathcal{F}^{cu}$ leaves, it comes that $m_x^{cu}$ has Lebesgue disintegration along $\mathcal{F}^{c}|_{\mathcal{F}^{cu}(x)}$, so $m$ has Lebesgue disintegration along $\mathcal{F}^{c}$.
    The same can be said for the subcenter case since the holonomy of $\mathcal{F}^{\hat{c}}|_{\mathcal{F}^{\hat{c}u}(x)}$ between two nearby points $y\in \mathcal{F}^{\hat{c}u}(x)$ and $z \in \mathcal{F}^{\hat{c}}_{\text{loc}}(y)$ with respect to the transversals $\mathcal{F}^u_{\text{loc}}(y)$ and $\mathcal{F}^u_{\text{loc}}(z)$ coincides with the holonomy of $\mathcal{F}^{cs}$ between $y$ and $z$.
\end{proof}

\section{Holomorphicity of the subcenter holonomy restricted to subcenter-unstable leaves}
\label{sec:5}
From now until the rest of the paper, unless mentioned otherwise, we will state the results we obtain in the case of a smooth transversely holomorphic partially hyperbolic flow on a smooth compact manifold $M$ of dimension $7$ with a flow invariant compact subcenter foliation $\mathcal{F}^{\hat{c}}$ with $C^1$ leaves and trivial holonomy.
Our results are analogous to that of the discrete holomorphic case, studied in \cite{xu_holomorphic_2025}, by considering the subcenter foliation instead of the center one, and the time-one map $\varphi^1$ instead of the holomorphic partially hyperbolic diffeomorphism $f$.
In the proof of our results, we only refer to the (minor) changes.

Throughout this section, $(\varphi^t)$ is a smooth transversely holomorphic partially hyperbolic flow on a smooth compact manifold $M$ of dimension $7$ whose subcenter distribution is integrable to a flow invariant compact subcenter foliation $\mathcal{F}^{\hat{c}}$ with $C^1$ leaves and trivial holonomy.
$h^{\hat{c}}$ will denote the subcenter holonomy along the unstable direction, and $\mu$ an ergodic Gibbs $u$-state. Let $\overline{M}=M/\mathcal{F}^{c}$ be the leaf space of the center foliation. We will always use $\overline{\cdot}$ symbols for quotient objects. For example, $\overline{\varphi^t}$ is the quotient map of $\varphi^t$ on $\overline{M}$, $\overline{\mu}$ is the quotient measure on $\overline{M}$ of a measure $\mu$ on $M$.\\
Let $t\in \mathbb R$ and $x,y\in M$ such that $\overline{x}=\overline{y}$, that is $y \in \mathcal{F}^{\hat{c}}(x)$. 
Then by flow invariance of the subcenter foliation, it comes $\overline{\varphi^t(x)}=\overline{\varphi^t(y)}$. 
Therefore, we will sometimes use the notation $\varphi^t(\overline{x})$ to mean the heavier $\overline{\varphi^t(x)}$. 
Also, $\mathcal{F}^{\hat{c}}(\overline{x})$ will mean $\mathcal{F}^{\hat{c}}(y)$ for any $y\in \overline{x}$.\\

Since the subcenter foliation $\mathcal{F}^{\hat{c}}$ is a compact $C^0$ foliation (the subcenter bundle is assumed to be a continuous subbundle of $TM$), Reeb's stability theorem (Theorem \ref{thm:reeb}) applies and we obtain thus the following result:

\begin{lemma}
    $M \xrightarrow[]{p}\overline{M}$ is a continuous fiber bundle whose fibers are exactly the subcenter leaves.
\end{lemma}

\subsection{A disjunction: isometry or contraction}

A subcenter-leaf-wise defined metric $g^{\hat{c}}$ (that is a family $(g^{\hat{c}}_{\overline{x}})_{\overline{x}\in \overline{M}}$ such that for $\overline{x}\in \overline{M}$, $g^{\hat{c}}_{\overline{x}}$ is a metric on $\mathcal{F}^{\hat{c}}(\overline{x})$) is called \emph{uniformly bounded} if there exists $C>0$ such that for any $\overline{x}\in\overline{M}$ and $x_{1},x_{2}\in \mathcal{F}^{\hat{c}}(\overline{x})$,
$d_{g_{\overline{x}}^{\hat{c}}}(x_{1},x_{2})\leq Cd^{\hat{c}}(x_{1},x_{2})$, where $d^{\hat{c}}(\cdot,\cdot)$ is the distance induced by the restriction of a Riemannian metric on $M$ to a subcenter leaf. 
This definition is independent of the choice of the Riemannian metric.

\begin{definition}
If there is a measurably subcenter-leaf-wise metric $g^{\hat{c}}$ on $M$ such that:
\begin{itemize}
    \item $g^{\hat{c}}$ is uniformly bounded;
    \item $\varphi^1|_{\mathcal{F}^{\hat{c}}(\overline{x})}$ is an isometry with respect to  $g^{\hat{c}}_{\overline{x}}$ and $g^{\hat{c}}_{\varphi^1(\overline{x})}$ for $\overline{\mu}$-almost every $\overline{x}\in\overline{M}$,
\end{itemize}
then $(\varphi^t)$ is called a \emph{$\mu$-subcenter isometry}.
\end{definition}

For any $x\in M$, we denote by $L^{\hat{c}u}_{\operatorname{loc}}(x)$
the $\mathcal{F}^{\hat{c}}$-saturated set containing $\mathcal{F}^{u}_{\operatorname{loc}}(x)$, that is 
$$L^{\hat{c}u}_{\operatorname{loc}}(x)=\bigcup_{z \in \mathcal{F}^{u}_{\operatorname{loc}}(x)}\mathcal{F}^{\hat{c}}(z)=\bigcup_{w \in \mathcal{F}^{\hat{c}}(x)}\mathcal{F}^{u}_{\operatorname{loc}}(w).$$

\begin{definition}
If there exists a subset $A\subset\operatorname{supp}\mu$ such that
\begin{itemize}
    \item $\liminf_{n\to+\infty}d^{\hat{c}}\left(\varphi^{n}(x),\varphi^{n}(x') \right)=0$ for any $x\in A$, $x'\in A\cap \mathcal{F}^{\hat{c}}(x)$;
    \item $A\cap L^{\hat{c}u}_{\operatorname{loc}}(x)$ has full (leafwise Riemannian) volume in $L^{\hat{c}u}_{\operatorname{loc}}(x)$ for $\mu$-a.e. $x$.
\end{itemize}
then $(\varphi^t)$  is called a \emph{$\mu$-subcenter contraction}.
\end{definition}

As has been proved in the holomorphic discrete case, we prove:

\begin{proposition}
\label{prop:isoorcont}
Let $\mu$ be an ergodic Gibbs $u$-state. Then exactly one of the following holds:
\begin{enumerate}[label=(\roman*)]
    \item $(\varphi^t)$ is a $\mu$-subcenter isometry;
    \item $(\varphi^t)$ is a $\mu$-subcenter contraction.
\end{enumerate}
\end{proposition}
By Corollary \ref{cor:Fçholoequiv}, all subcenter leaves are holomorphically equivalent. 
We denote by $p_g$ the genus of the subcenter leaves. Also, fix a metric on $M$ and consider the set
    $$\overline{B}:=\{\overline{x} \in \overline{B} \mid \|d(\varphi^n|_{\mathcal{F}^{\hat{c}}(x)})\| \text{ is uniformly bounded for $n \geq 0$}\}.$$
In fact we have the following result:
    \begin{proposition}
    Let $\mu$ be an ergodic Gibbs $u$-state.
    \begin{itemize}
        \item If $p_g\geq 1$, then $(\varphi^t)$ is a $\mu$-subcenter isometry ;
        \item If $p_g=0$, then:
        \begin{itemize}
            \item either $\overline{\mu}(\overline{B})=1$ and $(\varphi^t)$ is a $\mu$-subcenter isometry ;
            \item or $\overline{\mu}(\overline{B})=0$ and $(\varphi^t)$ is a $\mu$-subcenter contraction.
        \end{itemize}
    \end{itemize}
    \end{proposition}
\begin{proof}
    If $p_g\geq 1$, for every $\overline{x} \in \overline{M}$, $\varphi^1|_{\mathcal{F}^{\hat{c}}(\overline{x})}: \mathcal{F}^{\hat{c}}(\overline{x}) \to \mathcal{F}^{\hat{c}}(\overline{\varphi^1}(\overline{x}))$ is holomorphic by \ref{prop:leavesholo} so the rest of the proof is identical to that of Lemma 3.4 of \cite{xu_holomorphic_2025}.\\
    If $p_g=0$, the case $\overline{\mu}(\overline{B})=1$ is identical (i.e. Lemma 3.6 of \cite{xu_holomorphic_2025}), while the case $\overline{\mu}(\overline{B})=0$ (Lemma 3.7 of \cite{xu_holomorphic_2025}) uses Proposition \ref{prop:fubini} which is merely Fubini's theorem in the subcenter-unstable leaves.
\end{proof}
\subsection{Contraction: the subcenter holonomy is holomorphic for a.e. leaf}
The main result of this subsection is the following:
\begin{proposition}
\label{prop:contract1}
If $(\varphi^t)$ is a $\mu$-subcenter contraction, then for any $\overline{x}\in\operatorname{supp}\overline{\mu}$ and any $x,y\in \mathcal{F}^{\hat{c}}(\overline{x})$, the subcenter holonomy between unstable leaves $h_{xy}^{\hat{c}}:\mathcal{F}_{\operatorname{loc}}^{u}(x)\to \mathcal{F}_{\operatorname{loc}}^{u}(y)$ is holomorphic.
\end{proposition}
The proof of this proposition relies heavily on the following lemma, which is analogous to Lemma 4.2 of \cite{xu_holomorphic_2025}.

\begin{lemma}
\label{lem:contract1}
Let $(\varphi^t)$ a smooth transversely holomorphic partially hyperbolic flow on a smooth compact manifold $M$ of dimension $7$ with a flow invariant compact subcenter foliation $\mathcal{F}^{\hat{c}}$ with $C^1$ leaves and trivial holonomy.
Assume there exist $x\in M$ and $x' \in \mathcal{F}^{\hat{c}}(x)$ such that $h^{\hat{c}}_{x,x'}: \mathcal{F}^u_{\mathrm{loc}}(x) \to \mathcal{F}^u_{\mathrm{loc}}(x')$ is differentiable at $x$ and there is, for some $\delta >0$, a sequence of conformal maps $f_n: \mathcal{F}^u_{\delta}(\varphi^n (x)) \to M$ such that
\[
\liminf_{n \to +\infty} d_{C^0}\left(h^{\hat{c}}_{\varphi^n (x), \varphi^n (x')}, f_n \right) = 0.
\]
Then $\overline{\partial}\left( \theta_{x'}^{-1} \circ h^{\hat{c}}_{x,x'}\circ \theta_x \right)(0)=0.$
\end{lemma}
\begin{proof}
    See the proof of Lemma 4.2 of \cite{xu_holomorphic_2025}.
    We use the non-stationary linearization given by Proposition \ref{prop:non-statiolin} as well as the fact that $d\varphi^1|_{E^u}$ is a holomorphic thus conformal map.
\end{proof}

\begin{lemma}
\label{lem:contract2}
Let $(\varphi^t)$ a smooth transversely holomorphic partially hyperbolic flow on a smooth compact manifold $M$ of dimension $7$ with a flow invariant compact subcenter foliation $\mathcal{F}^{\hat{c}}$ with $C^1$ leaves and trivial holonomy. Let $x \in M$ and $x' \in \mathcal{F}^{\hat{c}}(x)$ such that $h_{xx'}^{\hat{c}}:\mathcal{F}_{\operatorname{loc}}^{u}(x)\to \mathcal{F}_{\operatorname{loc}}^{u}(x')$ is differentiable at $x$ and $\liminf_{n \to +\infty} d_{\hat{c}}(\varphi^n(x),\varphi^n(x'))=0$.\\
Then $\overline{\partial}\left( \theta_{x'}^{-1} \circ h^{\hat{c}}_{x,x'}\circ \theta_x \right)(0)=0$.
\end{lemma}
\begin{proof}
   Take $f_n$ to be the identity map on $\mathcal{F}^u_{\delta}(\varphi^n (x))$ and apply the previous Lemma.
\end{proof}

\begin{lemma}
    If $(\varphi^t)$ is a $\mu$-subcenter contraction, then for $\mu$-a.e. $x\in M$ and $\text{vol}_{\mathcal{F}^{\hat{c}}}$-a.e. $x' \in \mathcal{F}^{\hat{c}}(x)$, $h_{xx'}^{\hat{c}}:\mathcal{F}_{\operatorname{loc}}^{u}(x)\to \mathcal{F}_{\operatorname{loc}}^{u}(x')$ is holomorphic.
\end{lemma}
\begin{proof} 
See Lemma 4.4 of \cite{xu_holomorphic_2025}.
We use the Fubini theorem on the subcenter unstable leaves (Proposition \ref{prop:fubini}), as well as the fact that the subcenter holonomy between local unstable leaves is quasiconformal (by Proposition \ref{prop:hcsquasiconf}, since it coincides with the center-stable holonomy between local unstable leaves), thus absolutely continuous and differentiable almost everywhere (Lemma \ref{lem:quasiconfholo}). We also use the previous result as well as the properties of the coordinates charts $(\theta_x)_{x\in M}$ mentioned in \ref{prop:non-statiolin}.
\end{proof}

\begin{proof}[Proof of Proposition \ref{prop:contract1}]
By the previous Lemma, there exists a set $A\subset M$ of full $\mu$ measure for every $x \in A$, $h^{\hat{c}}_{x x'}$ is holomorphic for $\operatorname{vol}_{\mathcal{F}^{\hat{c}}}$-a.e. $x' \in \mathcal{F}^{\hat{c}}(x)$. Since a full volume subset of a compact set is necessarily dense, for every $x\in A$, and $y \in \mathcal{F}^{\hat{c}}(x)$ there exists a sequence $(y_n)$ converging to $y$ such that for every $n\in \mathbb N$, $x_n \in \mathcal{F}^{\hat{c}}(x)$ and $h_{xx_n}^{\hat{c}}$ is holomorphic. 
Since $h_{xx_n}^{\hat{c}}$ converges uniformly to $h_{xy}^{\hat{c}}$, the latter is holomorphic.
Also, by definition of the support of a measure, it comes that $A\cap \text{supp}(\mu)$ is dense in $\text{supp}(\mu)$. Therefore, for any $\hat{x} \in \operatorname{supp} \hat{\mu}$ and any $x,y \in \mathcal{F}^{\hat{c}}(\hat{x})$, we can find sequences $x_n \to x$, $y_n \to y$ with $y_n \in \mathcal{F}^{\hat{c}}(x_n)$ such that $h^{\hat{c}}_{x_n y_n}$ is holomorphic.
As a result, $h^{\hat{c}}_{xy}$ is holomorphic, which concludes.
\end{proof}

\subsection{Isometric case}
\subsubsection{Existence of a Gibbs $\hat{c}u$-state}
The main result of this subsection is:
\begin{proposition}
\label{prop:gibbsçu}
    Let $(\varphi^t)$ a smooth partially hyperbolic flow on a smooth compact manifold $M$ with a flow invariant compact subcenter foliation $\mathcal{F}^{\hat{c}}$ with $C^1$ leaves and trivial holonomy, and let $\mu$ a Gibbs u-state. Assume that $(\varphi^t)$ is a $\mu$-subcenter isometry and that $\mathcal{F}^{\hat{c}}$ is absolutely continuous within $\mathcal{F}^{\hat{c}u}$.\\
    Then there exists a Gibbs $\hat{c}u$-state $\nu$ whose support is the $\mathcal{F}^{\hat{c}}$-saturated subset of $M$ which projects to $\operatorname{supp}(\overline{\mu})$ in $\overline{M}$.
\end{proposition}
The proof of this result is analogous to that of Proposition 5.1 of \cite{xu_holomorphic_2025}. We recall as well some facts about the heat semigroup on manifolds.\\

Let $N$ a smooth Riemannian manifold without boundary.
Let $\text{vol}$ be the normalized volume induced by the Riemannian metric, and $\Delta$ be the associated Laplace-Beltrami operator. 
Then there is a semigroup action $(P_t)_{t\geq0}$, such that for $t\geq0$, $P_{t}:L^{2}(N,\operatorname{vol})\to L^{2}(N,\operatorname{vol})$ satisfies: for every $f\in L^{2}(N,\operatorname{vol})$, $u(t,x)=P_{t}(f)$ is the solution to the Cauchy problem:

\[
\left\{\begin{array}{l}
\frac{\partial u}{\partial t}=\Delta_{x}u,\quad t>0\\
u|_{t=0}=f.
\end{array}\right.
\]
In particular, for every $f\in L^{2}(N,\operatorname{vol})$, $P_0(f)=f$.\\
The semigroup $(P_t)$ satisfies the following properties. For every $t\geq0$:

\begin{itemize}
    \item $P_{t}(1)=1$ and for every non-negative function $f\in L^{2}(N)$, $P_{t}(f)(x)\geq 0$.
    
    \item If $\gamma:N\to N'$ is an isometry, then for every $f\in L^{2}(N')$
    \[
    (P_{N,t})(f\circ\gamma)=P_{N',t}(f)\circ\gamma.
    \]
    \item For every $f \in L^2(N)$ and $x \in N $,
    \[
    \lim_{t \to +\infty} P_t(f)(x) = \int_N f(y) \, d\operatorname{vol}(y).
    \]
\end{itemize}

From now on, assume $(\varphi^t)$ satisfies the assumptions of Proposition \ref{prop:gibbsçu}. Therefore, $\varphi^1|_{\mathcal{F}^{\hat{c}}(\overline{x})}$ is an isometry with respect to a subcenter-leafwise metric $g_{\overline{x}}$ for $\overline{\mu}$-a.e. $\overline{x} \in \overline{M}$. 
Let $P_{\overline{x}, t}$ be the heat semigroup on $\mathcal{F}^{\hat{c}}(\overline{x})$ with respect to $g_{\overline{x}}$. 
For every $t \geq 0$ and any Borel probability measure $\mu$, by the previous properties of the heat semi group and Riesz representation theorem, there exists a unique Borel probability measure $\mu_t$ such that for any bounded measurable function $f$:
\[
\int_M f \, d\mu_t = \int_{\overline{M}} \left( \int_{\mathcal{F}^{\hat{c}}(\overline{x})} P_{\overline{x}, t}(f|_{\mathcal{F}^{\hat{c}}(\overline{x})}) \, d\mu_{\overline{x}}^{\hat{c}}\right) d\overline{\mu}(\overline{x}),
\]
where $\mu_{\overline{x}}^{\hat{c}}$ is the conditional measure of $\mu$ along $\mathcal{F}^{\hat{c}}(x)$, and $\overline{\mu}$ is the projection of $\mu$ to $\overline{M}$. 
By Proposition \ref{prop:fubini}, it comes $\mu_0=\mu$. Also:
\begin{lemma}
Let $\mu$ a Borel probability measure on $M$.\\
Then for every $t \geq 0$, $\overline{\mu_t}=\overline{\mu}$ and $\mu_t$ is $\varphi^1$-invariant.
\end{lemma}
\begin{proof}
    Let $A$ a measurable subset of $\overline{M}$ and $t\geq 0$.
    Then:
    \begin{align*}
        \mu_t(p^{-1}(A))= \int_{A} \left( \int_{\mathcal{F}^{\hat{c}}(\overline{x})} P_{\overline{x}, t}(1) \, d\mu_{\overline{x}}^{\hat{c}}\right) d\overline{\mu}(\overline{x})=\int_{\overline{M}} \left( \int_{\mathcal{F}^{\hat{c}}(\overline{x})} \mathrm{1}_{p^{-1}(A)}\, d\mu_{\overline{x}}^{\hat{c}}\right) d\overline{\mu}(\overline{x})=\mu(p^{-1}(A)),
    \end{align*}
    which proves exactly the first point.
As for the $\varphi^1$-invariance of $\mu_t$, the proof is analogous to that of Lemma 5.2 of \cite{xu_holomorphic_2025} by replacing $f$ with $\varphi^1$.
\end{proof}

For $ x \in M $, let again 
\[
L_{\operatorname{loc}}^{\hat{c}u}(x):= \mathcal{F}^{\hat{c}}(\mathcal{F}_{\operatorname{loc}}^{u}(x))=\bigcup_{z \in \mathcal{F}^u_{\operatorname{loc}}(x)}\mathcal{F}^{\hat{c}}(z)
\] 
the $ \mathcal{F}^{\hat{c}} $-saturated set of $ \mathcal{F}_{\operatorname{loc}}^{u}(x)$.
By Proposition 6.12 of \cite{abouanass_dynamical_2026},
the map
$$\left \{\begin{array}{ccl}
        \mathcal{F}^u_{\operatorname{loc}}(x) \times \mathcal{F}^{\hat{c}}(x)&\to & M\\
         (y,z) &\mapsto & \mathcal{F}^{\hat{c}}(y)\cap \mathcal{F}^u_{\operatorname{loc}}(z)
    \end{array} \right.$$
    is a well-defined homeomorphism onto its image, so
$L_{\operatorname{loc}}^{\hat{c}u}(x) $ is homeomorphic to $ \mathcal{F}^{\hat{c}}(x) \times \mathcal{F}_{\operatorname{loc}}^{u}(x)$ through local unstable and subcenter holonomy.
Denote by $\Pi^{\hat{c}}$ and $\Pi^{u}$ the projections from $L_{\operatorname{loc}}^{\hat{c}u}(x)$ to $ \mathcal{F}^{\hat{c}}(x)$ along local unstable leaves and to $ \mathcal{F}_{\operatorname{loc}}^{u}(x) $ along subcenter leaves respectively.
\begin{lemma}
For any $ t \geq 0 $, the local disintegration $ \mu_{t,x}^{\hat{c}u} $ of $ \mu_t $ on $  L_{\operatorname{loc}}^{\hat{c}u}(x) $ has a product structure with respect to the topological product structure $  L_{\operatorname{loc}}^{\hat{c}u}(x) = \mathcal{F}^{\hat{c}}(x) \times\mathcal{F}_{\operatorname{loc}}^{u}(x) $, i.e. for $\mu_t$-a.e. $x\in M$:
\[
\mu_{t,x}^{\hat{c}u} = \mu_{t,x}^{\hat{c}} \times \Pi_{*}^{u} \mu_x^{\hat{c}u}.
\]
\end{lemma}
\begin{proof}
    We will need the following result, which comes from Corollary 4.3 of \cite{avila_extremal_2010}, since $(\varphi^t)$ is a $\mu$-subcenter isometry and therefore its subcenter Lyapunov exponents vanish. 
    \begin{proposition}
    Let $(\varphi^t)$ a smooth partially hyperbolic flow on a smooth compact manifold $M$ with a flow invariant compact subcenter foliation $\mathcal{F}^{\hat{c}}$ with $C^1$ leaves and trivial holonomy, and let $\mu$ be a $\varphi^1$-invariant Borel probability measure. Assume that $(\varphi^t)$ is a $\mu$-subcenter isometry.\\
    Then $\mu$ admits a disintegration $(\mu^{\hat{c}}_x)_{x \in M}$ along the subcenter foliation such that for $\mu$-a.e. $x \in M$ and every $y \in \mathcal{F}^u(x)$, 
    $$(h^u_{x,y})_*\mu^{\hat{c}}_x=\mu^{\hat{c}}_y.$$
\end{proposition}
By applying this Proposition to $\mu_t$ for $t \geq 0$, it comes that $\mu_t$ admits a disintegration which is invariant by unstable holonomies on a full $\mu_t$-measure set $E_t$, i.e., for any $x \in E_t$ and $y \in \mathcal{F}^u(x)$,
\[
(h_{xy}^u)_*\mu_{t,x}^{\hat{c}} = \mu_{t,y}^{\hat{c}}.
\]
The rest of the proof is identical to that of Lemma 5.4 of \cite{xu_holomorphic_2025}, by replacing "center" with "subcenter".
\end{proof}
\begin{corollary}
    For $\mu$-a.e. $x\in M$, $\Pi^u_*\mu^{\hat{c}u}_x$ is equivalent to the leafwise Lebesgue measure along $\mathcal{F}^u_{\operatorname{loc}}(x)$.
\end{corollary}
\begin{proof}
The proof is similar to that of Corollary 5.5.
We use here that $h^{\hat{c}}$ is absolutely continuous between local unstable leaves as well as the previous results.
\end{proof}
Let $\nu$ a weak* limit point of $(\mu_t)_{t\geq0}$.

\begin{lemma}
   $\nu$ \emph{is a Gibbs $cu$-state with $\operatorname{supp}\nu = \{x \in M \mid \overline{x} \in \operatorname{supp}\overline{\mu}\}$.}
\end{lemma}

\begin{proof}
The proof is similar to that of Corollary 5.6.
We use the previous results as well as Fubini's theorem on local subcenter-unstable leaves (Proposition \ref{prop:fubini}).
\end{proof}

\subsubsection{The subcenter holonomy is $\mathbb R$-linear for a.e. leaf}
We will need a serie of Lemmas in order to prove the main result of this subsection, whose proofs are given in the holomorphic discrete case in Appendix B of \cite{xu_holomorphic_2025}:
\begin{proposition}
\label{prop:subcentholoRlin}
    Let $\nu$ be a Gibbs $\hat{c}u$ state.\\
    Then for every $x,y \in \operatorname{supp}(\nu)$ such that $y \in \mathcal{F}^{\hat{c}}(x)$, the map $H^{\hat{c}}_{x,y}:=\theta_y^{-1} \circ h^{\hat{c}}_{x,y}\circ \theta_x$ is an $\mathbb R$-linear map.
    In particular, $h^{\hat{c}}_{x,y}$ is defined on $\mathcal{F}^u(x)$.
\end{proposition}

We consider the space
\[
\mathcal{E}=\{(x,y)\in M^{2}:x\in M,\; y\in \mathcal{F}^{\hat{c}}(x)\},
\]
and we define the map $\varphi_{\mathcal{E}}:\mathcal{E}\to\mathcal{E},\quad \varphi_{\mathcal{E}}(x,y)=(\varphi^1(x),\varphi^1(y))$.

\begin{lemma}
The space $\mathcal{E}$ is a continuous fiber bundle over $M$ with compact fibers. Moreover $\varphi_{\mathcal{E}}$ preserves an invariant probability measure $\nu_{\mathcal{E}}$ on $\mathcal{E}$ such that for any continuous function $f$ on $\mathcal{E}$ ,
\[
\int_{\mathcal{E}}f \, d\nu_{\mathcal{E}}=\int_{M}\int_{\mathcal{F}^{\hat{c}}(x)}f(x,y)\, d\nu^{\hat{c}}_{x}(y)\, d\nu(x). \]
\end{lemma}
\begin{proof}
    We replace, in \cite{butler_uniformly_2018} Proposition 20, the assumption that $f$ preserves the volume with the existence of a Gibbs $\hat{c}u$-state $\nu$ given by Proposition \ref{prop:fubini}.
\end{proof}

\begin{lemma}
Let
\[
Q:=\{x\in M:\text{for } \nu^{\hat{c}}_{x}\text{-a.e. } y\in \mathcal{F}^{\hat{c}}(x),\; h^{\hat{c}}_{xy}:\mathcal{F}^{u}_{\operatorname{loc}}(x)\to \mathcal{F}^{u}_{\operatorname{loc}}(y) \text{ is differentiable at } x\}.
\]
Then $\nu(Q)=1$.
\end{lemma}
\begin{proof}
    The proof is similar to that of Lemma B.2 of \cite{xu_holomorphic_2025} by replacing "center" with "subcenter" and using Fubini's theorem as well as the fact that $\nu$ is a Gibbs $\hat{c}u$-state.
\end{proof}
Let $ \mathcal{Q} = \{(x, y) \in \mathcal{E}: x \in Q\} $. Then $ \mathcal{Q} $ has full $ \nu_{\mathcal{E}} $-measure inside $ \mathcal{E} $ by the definition of $Q$ and $\nu_{\mathcal{E}}$. For $(x, y) \in \mathcal{Q} $, we define $ H_{xy}^{\hat{c}}:=d_xh_{xy}^{\hat{c}}: E_{x}^{u} \to E_{y}^{u} $ to be the derivative of $ h_{xy}^{\hat{c}} $ at $ x $. The map $(x, y) \to H_{xy}^{\hat{c}} $ is measurable and defined $ \nu_{\mathcal{E}} $-a.e.

\begin{lemma}
There is a full $ \nu_{\mathcal{E}} $-measure subset $ \mathcal{Q'} \subset \mathcal{Q} $ such that if $(x, y), (z, w) \in \mathcal{Q'} $ with $ z \in \mathcal{F}_{\operatorname{loc}}^{u}(x) $ and $ w \in \mathcal{F}_{\operatorname{loc}}^{u}(y) $, then
\[
H_{zw}^{\hat{c}} \circ H_{xz}^{u} = H_{yw}^{u} \circ H_{xy}^{\hat{c}}.
\]
\end{lemma}
\begin{proof}
    The proof is similar to that of Lemma B.4 in \cite{xu_holomorphic_2025}. We replace $f$ with $\varphi^1$, use Proposition \ref{prop:Hu} as well as the equality for $n \in \mathbb N$, $x \in M$ and $y\in \mathcal{F}^{\hat{c}}(x)$:
    
    $h_{xy}^{\hat{c}}|_{\mathcal{F}^{u}_{\text{loc}(x)}}=\varphi^{-n}\circ h_{\varphi^{n}(x),\varphi^{n}(y)}^{\hat{c}}\circ \varphi^{n}|_{\mathcal{F}^{u}_{\text{loc}(y)}}.$
\end{proof}
Eventually, we come to the result:
\begin{lemma}
The measurable function $ H^{\hat{c}} $ on $ \mathcal{Q'} $ admits a continuous extension to $ \operatorname{supp} \nu_{\mathcal{E}} $:
for every $x, y \in \operatorname{supp} \nu $ with $ y \in \mathcal{F}^{\hat{c}}(x) $,
\[
\theta_y^{-1} \circ h^{\hat{c}}_{x,y} \circ \theta_x = H^{\hat{c}}_{x,y}.
\]
Therefore, the center holonomy is $\mathbb R$-linear in the charts $ \{\theta_x\}_{x \in M} $.
\end{lemma}
\begin{proof}
    The proof is similar to that of Lemma B.5 of \cite{xu_holomorphic_2025}.
    We use quasiconformality of the subcenter holonomy (see Proposition \ref{prop:hcsquasiconf}) as well as the previous Lemmas in order to conclude.
\end{proof}

\subsubsection{The subcenter holonomy is holomorphic for a.e. leaf}
We come to the main result concerning the subcenter-isometry case:
\begin{proposition}
\label{prop:isometryholo}
If $(\varphi^t)$ is a $\mu$-subcenter isometry, then for any $\overline{x}\in\operatorname{supp}\overline{\mu}$ and any $x,y\in \mathcal{F}^{\hat{c}}(\overline{x})$, the subcenter holonomy between unstable leaves $h_{xy}^{\hat{c}}:\mathcal{F}_{\operatorname{loc}}^{u}(x)\to \mathcal{F}_{\operatorname{loc}}^{u}(y)$ is holomorphic.
\end{proposition}

The proof of this result is again identical to that of the discrete case.
We can define on each unstable leaves a translation structure thanks to Proposition \ref{prop:non-statiolin}, in the following way.
Let $x\in M$ and $v \in E^u_x$. We define the map
$$T_{v}:\mathcal{F}^{u}(x)\to\mathcal{F}^{u}(x),\quad p\mapsto\theta_{x}(\theta^{-1}_{x}(p)+v)$$
which is a holomorphic diffeomorphism of $\mathcal{F}^{u}(x)$.
In our case, we have the following result (see Proposition 6.13 of \cite{abouanass_dynamical_2026}):
\begin{lemma}
     For every $x \in M$, the projection
    $$p\ \mapsto \mathcal{F}^{u}(p)\cap \mathcal{F}^{\hat{c}}(x)$$
    from $\mathcal{F}^{\hat{c}u}(x)$ to $\mathcal{F}^{\hat{c}}(x)$
    is a well-defined map, which is moreover a holomorphic submersion.
\end{lemma}
\begin{proof}
    Its existence and the fact that it is a submersion has been proved in \cite{abouanass_dynamical_2026}, since $(\varphi^t)$ is subcenter-bunched.. The holomorphicity comes from Proposition \ref{cor:FçholoinFçs} since in our case the unstable foliation is a holomorphic foliation of $\mathcal{F}^{\hat{c}u}(x)$.
\end{proof}

Therefore, thanks also to Proposition \ref{prop:subcentholoRlin}, we can extended $T_v$ to a homeomorphism of $\mathcal{F}^{\hat{c}u}(x)$:
\[
T_{v}:\mathcal{F}^{\hat{c}u}(x)\to\mathcal{F}^{\hat{c}u}(x),\quad p\mapsto\theta_{y}(\theta^{-1}_{y}(p)+H^{\hat{c}}_{xy}(v))
\]
where $y=\mathcal{F}^{u}(p)\cap \mathcal{F}^{\hat{c}}(x)$.
By definition, $T_{v}$ is a translation when restricted to an unstable leaf and is thus holomorphic. 
Moreover, we prove that:

\begin{lemma}
For any $v\in E^{u}_{x}$, $T_{v}:\mathcal{F}^{\hat{c}u}(x)\to\mathcal{F}^{\hat{c}u}(x)$ is a holomorphic diffeomorphism.
\end{lemma}
\begin{proof}
    By Holomorphic's Journé Lemma (Proposition 2.11 of \cite{xu_holomorphic_2025}) since the restriction of $T_v$ to some unstable leaf is of the form $T_{v'}$ and is thus holomorphic, we prove that its restriction to any subcenter leaf is holomorphic in order to conclude.
    In fact, one can prove analogously to Lemma 7.2 of \cite{xu_holomorphic_2025} and thanks to Propositions \ref{lem:globalholo} and \ref{prop:subcentholoRlin}, that for $p \in \mathcal{F}^{\hat{c}u}(x)$, 
    $T_v|_{\mathcal{F}^{\hat{c}}(p)}$ is equal to $h^u_{p, T_v(p)}$ the unstable holonomy in $\mathcal{F}^{\hat{c}u}(x)$ between subcenter leaves.
    By Corollary \ref{cor:FçholoinFçs}, it is holomorphic and thus we conclude.
\end{proof}

Fix two $\mathbb{R}$-linearly independent vectors $v_1, v_2 \in E^{u}(x)$, and let $\Lambda_x = \operatorname{span}_{\mathbb{Z}}\{v_1, v_2\}$ be a lattice in $E^{u}_{x}$.
Then $\{T_v \mid v \in \Lambda_x\}$ gives a properly discontinuous holomorphic action of $\Lambda_x$ on $\mathcal{F}^{\hat{c}u}(x)$. Let $N =\mathcal{F}^{\hat{c}u}(x)/\Lambda_x$ be the quotient manifold.

\begin{lemma}
There is a holomorphic submersion $\pi_N: N \to \mathcal{F}^{\hat{c}}(x)$ such that for any $y \in \mathcal{F}^{\hat{c}}(x)$, $\pi_N^{-1}(y)$ is a complex torus of dimension 1 .
\end{lemma}
\begin{proof}
   The proof is analogous to that of Lemma 7.3 of \cite{xu_holomorphic_2025}, with the help of the previous lemma.
\end{proof}
\begin{proof}[Proof of Proposition \ref{prop:isometryholo}]
    The proof is similar to that of Proposition 7.1 in \cite{xu_holomorphic_2025} with the help of the previous results, and by replacing "center" by "subcenter".
\end{proof}

\subsection{Conclusion}
Eventually, we come to the main result of this section:
\begin{proposition}
\label{prop:concholo}
    For every $x \in M$ and $y \in \mathcal{F}^{\hat{c}}(x)$, the subcenter holonomy between unstable leaves $h_{xy}^{\hat{c}}:\mathcal{F}_{\operatorname{loc}}^{u}(x)\to \mathcal{F}_{\operatorname{loc}}^{u}(y)$ is holomorphic.
\end{proposition}
We will need the following Lemma:
\begin{lemma}[8.2 of \cite{xu_holomorphic_2025}]
\label{lem:S}
Let $ f: M \to M $ be a $ C^2 $ partially hyperbolic diffeomorphism and $ S \subset M $ be an $ f $-invariant open set such that for every ergodic Gibbs $ u $-state $ \mu $, $\operatorname{supp} \mu \subset S$. \\
Then for every $ x \in M $, $ S \cap \mathcal{F}^{u}_{\mathrm{loc}}(x) $ has full $ m^{u}_{x} $-measure.
\end{lemma}
    We define the set
    $$S:=\{x \in M \mid \exists \delta=\delta(\overline{x})>0, \; \forall x_1,x_2 \in \mathcal{F}^{\hat{c}}(x),\, h^{\hat{c}}_{x_1,x_2}: \mathcal{F}^u_{\delta}(x_1) \to \mathcal{F}^u_{\delta}(x_2) \text{ is holomorphic}\}. $$
    We prove that $S$, together with $f=\varphi^1$, satisfies the assumptions of the previous Lemma.
    By construction, $S$ is $\mathcal{F}^{\hat{c}}$-saturated. Also, since for every $x \in M$ and $x_1,x_2 \in \mathcal{F}^{\hat{c}}(x)$, 
    $$h^{\hat{c}}_{\varphi^1(x_1),\varphi^1(x_2)}=\varphi^1\circ h^{\hat{c}}_{x_1,x_2} \circ \varphi^{-1}|_{\mathcal{F}^u_{\operatorname{loc}}(\varphi^1(x_1))},$$
    it comes that $S$ is $\varphi^1$-invariant.
\begin{lemma}
    For any ergodic Gibbs $u$-state $\mu$, we have $\operatorname{supp} \mu \subset S$.
\end{lemma}
\begin{proof}
The proof is similar to Lemma 8.3 \cite{xu_holomorphic_2025} since we have proven that:
\begin{itemize}
    \item $ (\varphi^t) $ is either a $\mu$-subcenter isometry or a $\mu$-subcenter contraction.
    \item If $ (\varphi^t) $ is a $\mu$-subcenter contraction, then for every $\overline{x} \in \operatorname{supp} \overline{\mu}$ and every $ x_1, x_2 \in \mathcal{F}^{\hat{c}}(\overline{x}) $, $h_{x_1 x_2}^{\hat{c}}: \mathcal{F}_{\mathrm{loc}}^{u}(x_1) \to \mathcal{F}_{\mathrm{loc}}^{u}(x_2) $ is holomorphic (see Proposition \ref{prop:contract1}).
    \item If $ (\varphi^t) $ is a $\mu$-subcenter isometry, then for every $\overline{x} \in \operatorname{supp} \overline{\mu}$ and every $ x_1, x_2 \in \mathcal{F}^{\hat{c}}(\overline{x}) $, $h_{x_1 x_2}^{\hat{c}}: \mathcal{F}_{\mathrm{loc}}^{u}(x_1) \to \mathcal{F}_{\mathrm{loc}}^{u}(x_2) $ is holomorphic (see Proposition \ref{prop:isometryholo}).
\end{itemize}
\end{proof}
\begin{lemma}
    $S$ is open.
\end{lemma}
\begin{proof}
    Since $S$ is saturated by subcenter leaves, we prove that $\overline{S} \subset \overline{M}$ is open.
    By Proposition 7.1 of \cite{abouanass_dynamical_2026}, it suffices to show that for every $\overline{x} \in \overline{S}$, $\overline{\mathcal{F}^u_{\text{loc}}(x) }\subset \overline{S}$ and $\overline{\mathcal{F}^{ws}_{\text{loc}}(x)} \subset \overline{S}$.
    Fix $x \in S$. It follows from the definition of $S$ that $\overline{\mathcal{F}^u_{\text{loc}}(x)} \subset \overline{S}$.
    Also, since for every $t\in \mathbb R$, $x_1 \in M$ and $x_2 \in \mathcal{F}^{\hat{c}}(x_1)$, 
    $$h_{x_1,x_2}^{\hat{c}}=\varphi^{-t}\circ h_{\varphi^{t}(x),\varphi^{t}(y)}^{\hat{c}}\circ \varphi^{t}|_{\mathcal{F}^{u}_{\text{loc}}(x_1)},$$
    it comes that $\overline{\Phi_{\text{loc}}(x)} \subset \overline{S}$.
    By Lemma \ref{lem:metricweak}, it is enough to prove that $\overline{\mathcal{F}^s_{\text{loc}}(x)} \subset \overline{S}$ to conclude.
    The proof of this point is similar to that of Lemma 8.4 in \cite{xu_holomorphic_2025}, by replacing "center" by "subcenter" and $f$ with $\varphi^1$.
\end{proof}
\begin{proof}[Proof of Proposition \ref{prop:concholo}]
Let $x_1 \in M$ and $ x_2 \in \mathcal{F}^{\hat{c}}(x_1)$. By definition of $S$, it comes that 
\[
h^{\hat{c}}_{x_1,x_2}: \mathcal{F}^u_{\mathrm{loc}}(x_1) \to \mathcal{F}^u_{\mathrm{loc}}(x_2)
\] 
is holomorphic on $S \cap \mathcal{F}^u_{\mathrm{loc}}(x_1)$. 
By Lemma \ref{lem:S}, $S \cap \mathcal{F}^u_{\mathrm{loc}}(x_1)$ has full $m^u_{x_1}$-measure. By quasiconformality of subcenter holonomy (Proposition \ref{prop:hcsquasiconf}) and Lemma \ref{lem:quasiconfholo}, it follows that $h^{\hat{c}}_{x_1,x_2}$ is holomorphic on $\mathcal{F}^u_{\mathrm{loc}}(x_1)$, which concludes the proof of Proposition \ref{prop:concholo}.
\end{proof}
\begin{corollary}
    For every $p \in M$, $\restr{\mathcal{F}^{\hat{c}}}{\mathcal{F}^{\hat{c}u}(p)}$ is a holomorphic foliation in $\mathcal{F}^{\hat{c}u}(p)$ and $\restr{\mathcal{F}^{\hat{c}}}{\mathcal{F}^{\hat{c}s}(p)}$ is a holomorphic foliation in $\mathcal{F}^{\hat{c}s}(p)$.
\end{corollary}
\begin{proof}
    This is a consequence of the previous Proposition as well as Proposition \ref{cor:FçholoinFçs} and Proposition \ref{prop:journeholomo}. 
\end{proof}

\section{Smoothness of the foliations and classification}
\label{sec:6}
\subsection{Smooth and transverse holomorphic structure of the center foliation}
\begin{proposition}
\label{prop:Fcstrholo}
    Let $(\varphi^t)$ a smooth transversely holomorphic partially hyperbolic flow on a smooth compact manifold $M$ of dimension $7$ with a flow invariant compact subcenter foliation $\mathcal{F}^{\hat{c}}$ with $C^1$ leaves and trivial holonomy.\\
    Then the center-stable $\mathcal{F}^{cs}$, center-unstable $\mathcal{F}^{cu}$ and center $\mathcal{F}^c$ foliations are $C^\infty$ and transversely holomorphic.
\end{proposition}
\begin{proof}
    We prove that the center-stable foliation $\mathcal{F}^{cs}$ is transversely holomorphic.
    Let $x \in M$ and $y \in \mathcal{F}^{cs}_{\text{loc}}(x)$.
    We want to show that $h^{cs}_{x,y}: \mathcal{F}^u_{\text{loc}}(x) \to \mathcal{F}^u_{\text{loc}}(y)$ is holomorphic.
    As we mentioned in the proof of Proposition \ref{prop:hcsquasiconf}, we can assume $y \in \mathcal{F}^{\hat{c}s}_{\text{loc}}(x)$.
    Let $z$ be the unique intersection point of $\mathcal{F}^{\hat{c}}_{\text{loc}}(x)$ with $\mathcal{F}^{s}_{\text{loc}}(y)$ and define 
    $$H^{cs}_{x,y}:=H^{s}_{z,y}\circ d_xh^{\hat{c}}_{x,z}$$
    where $H^s$ is the stable holonomy of the cocycle $d\varphi^1|_{E^u}$ (see Proposition \ref{prop:Hu}). By the remark after that proposition, and Proposition \ref{prop:concholo}, $H^{cs}_{x,y}$ is $\mathbb C$-linear and:
    \begin{lemma}
        Let $(\varphi^t)$ a smooth transversely holomorphic partially hyperbolic flow on a smooth compact manifold $M$ of dimension $7$ with a flow invariant compact subcenter foliation $\mathcal{F}^{\hat{c}}$ with $C^1$ leaves and trivial holonomy.\\
        Then for $x \in M$ and $y\in \mathcal{F}^{\hat{c}s}(x)$, the map $h^{cs}_{x,y}: \mathcal{F}^u_{\text{loc}}(x) \to \mathcal{F}^u_{\text{loc}}(y)$ is holomorphic, and whose differential at $x$ is $H^{cs}_{x,y}$.
    \end{lemma}
    \begin{proof}
        We first show that if $y \in \mathcal{F}^s_{\text{loc}}(x)$ and $h^{cs}_{x,y}$
        is differentiable at $x$, then necessarily $d_xh^{cs}_{x,y}=H^s_{x,y}$:
        the proof given in Lemma 28 of \cite{butler_uniformly_2018} in the case of a volume preserving $u$-quasiconformal partially hyperbolic diffeomorphism applies in our setting, by replacing $f$ with $\varphi^1$ and "center" with "subcenter".
        Now assume $x \in M$ and $y\in \mathcal{F}^{\hat{c}s}(x)$. Let $z$ be the unique intersection point of $\mathcal{F}^{\hat{c}}_{\text{loc}}(x)$ with $\mathcal{F}^{s}_{\text{loc}}(y)$.
        Then $h^{cs}_{x,y}=h^{cs}_{z,y}\circ h^{cs}_{x,z}$.
        The map $h^{cs}_{x,z}: \mathcal{F}^u_{\text{loc}}(x) \to \mathcal{F}^u_{\text{loc}}(z)$ coincides with the subcenter holonomy between $\mathcal{F}^u_{\text{loc}}(x)$ and $\mathcal{F}^u_{\text{loc}}(z)$ in $\mathcal{F}^{\hat{c}u}(x)$ which is holomorphic by Proposition \ref{prop:concholo} thus $C^1$.
        Assume $h^{cs}_{x,y}$ is differentiable at $x$.
        Then by the above, $h^{cs}_{z,y}$ is differentiable at $z$ and its derivative at $z$ is $H^{s}_{z,y}$.
        Therefore the differential of $h^{cs}_{x,y}$ at $x$ is $H^{s}_{z,y}\circ d_xh^{\hat{c}}_{x,z}=H^{cs}_{x,y}$ which is $\mathbb C$-linear.
        By Proposition \ref{lem:quasiconfholo}, we know that $h^{cs}_{x,y}$ is differentiable $m^u_x$-almost everywhere.
        Since the restriction to $h^{cs}_{x,y}$ to an open neighborhood of $z\in \mathcal{F}^u_{\operatorname{loc}}(x)$ in $\mathcal{F}^u_{\operatorname{loc}}(x)$ is exactly of the form $h^{cs}_{z,z'}$, where $z'=h^{cs}_{x,y}(z)$, the above shows that for those points at which $h^{cs}_{x,y}$ is differentiable, the differential is $\mathbb C$-linear, i.e. $h^{cs}_{x,y}$ is holomorphic at every one of those points.
        By the second point of Proposition \ref{lem:quasiconfholo}, this proves that $h^{cs}_{x,y}: \mathcal{F}^u_{\text{loc}}(x) \to \mathcal{F}^u_{\text{loc}}(y)$  is holomorphic.
    \end{proof}
    We conclude the proof of Proposition \ref{prop:Fcstrholo}.
    By Proposition \ref{prop:caractrholo}, the previous Lemma shows that $\mathcal{F}^{cs}$ is transversely holomorphic.
    In particular, the center-stable holonomy between local unstable leaves is uniformly $C^\infty$ (see Proposition \ref{prop:unifholo}).
    Since $\mathcal{F}^{cs}$ coincides with the weak subcenter-stable foliation and the subcenter-stable leaves are holomorphic (by Proposition \ref{prop:leavesholo}) thus uniformly $C^\infty$ (see Proposition \ref{prop:unifleaves}), the center-stable leaves are uniformly $C^\infty$ (see the proof of Proposition 4.1 in \cite{abouanass_dynamical_2026} : if $\mathcal{F}$ has uniformly $C^r$ leaves, for $r\geq 1$ or $r=\infty$, then $\mathcal{F}^w$ also).
    By Journé's Lemma (Proposition \ref{prop:journe}), $\mathcal{F}^{cs}$ is a $C^\infty$ foliation.
    If we apply all our arguments to the flow $(\varphi^{-t})$, then we conclude in the same way that $\mathcal{F}^{cu}$ is transversely holomorphic with uniformly smooth leaves, thus smooth.\\
    Therefore by Definition \ref{def:trholo}, there exist a smooth atlas of submersions $(U_i, f_i^s)_{i\in I}$ and holomorphic transition maps $\gamma^s_{ij}$ defining the transverse holomorphic structure of the center-stable foliation. There also exist an atlas of submersions $(V_i, f_i^u)_{i\in J}$ and holomorphic transition maps $\gamma^u_{ij}$ defining the transverse holomorphic structure of the center-unstable foliation. 
    By taking a refinement if necessary, we can assume $I=J$ and $V_i=U_i$ for every $i\in I$.
    We define: $f_i = (f_i^u, f_i^s):U_i \to \mathbb C^2$ and 
    $\gamma_{ij}=\gamma_{ij}^u \otimes \gamma_{ij}^s:(z_1,z_2)\mapsto (\gamma_{ij}^u (z_1),  \gamma_{ij}^s(z_2))$.
    By Osgood's lemma, $\gamma_{ij}$ is holomorphic because it is continuous and holomorphic along each variable. 
    Also, since for every $p\in U_i$
    $$\ker(df_i^s)\cap\ker(df_i^u)=E^{cs}_p\cap E^{cu}_p =E^c_p,$$
    $f_i: U_i \to \mathbb C^4$ is a submersion.
    This defines an atlas of $C^\infty$ submersions $(U_i, f_i)$ onto $\mathbb C^4$ with holomorphic transition maps $\gamma_{ij}$ which is compatible with the one defining the center foliation. 
    This proves that $\mathcal{F}^c$ is also a smooth transversely holomorphic foliation.
\end{proof}
\begin{remark}
    The results obtained in the previous sections are also true when $E^{\hat{c}}$ is trivial, i.e. when $(\varphi^t)$ is a transversely holomorphic Anosov flow.
    The previous result says that the weak stable and weak unstable foliations of a transversely holomorphic Anosov flow on a smooth compact five-manifold are transversely holomorphic and $C^\infty$, which is what we obtained in \cite{abouanass_global_2025} more directly.
\end{remark}

\subsection{Smoothness of the subcenter foliation}
\begin{corollary}
    \label{cor:Fçsmooth}
    Let $(\varphi^t)$ a smooth transversely holomorphic partially hyperbolic flow on a smooth compact manifold $M$ of dimension $7$ with a flow invariant compact subcenter foliation $\mathcal{F}^{\hat{c}}$ with $C^1$ leaves and trivial holonomy.\\
    Then the sub-center foliation $\mathcal{F}^{\hat{c}}$ is $C^\infty$. 
\end{corollary}
\begin{proof}
    Since the subcenter leaves are holomorphic by Proposition \ref{prop:leavesholo} thus uniformly $C^\infty$, it suffices to show that the subcenter holonomy is uniformly $C^\infty$ (see Proposition \ref{prop:journe} again).
    Let $x,x'$ two nearby points on the same subcenter leaf. 
    Let $T_x$ (respectively $T_{x'}$) a small transversal to $\mathcal{F}^c$ at $x$ (at $x'$ respectively) containing $\mathcal{F}^{u}_{\delta_0}(x)$ and $\mathcal{F}^{s}_{\delta_0}(x)$ (containing $\mathcal{F}^{u}_{\delta_0}(x')$ and $\mathcal{F}^{s}_{\delta_0}(x')$ respectively) for some small $\delta_0$.
    Since the center-stable and the center-unstable foliations are transversely holomorphic, $T_x$ is a complex manifold subfoliated by two transverse transversely holomorphic $C^\infty$ foliations $\mathcal{F}^s_{T_x}$ and $\mathcal{F}^u_{T_x}$ obtained by intersecting respectively $\mathcal{F}^{cs}$ and $\mathcal{F}^{cu}$ with $T_x$.
    In the same way, $T_{x'}$ is a complex manifold subfoliated by two transverse transversely holomorphic $C^\infty$ foliations $\mathcal{F}^s_{T_{x'}}$ and $\mathcal{F}^u_{T_{x'}}$.
    There exists $0<\delta<\delta_0$ and $C>1$ such that the map
    $$\eta_x:\left\{\begin{array}{ccl}
        \mathcal{F}^{u}_{\delta}(x)\times \mathcal{F}^{s}_{\delta}(x) & \to & T_x\\
        (y,z) & \mapsto & \mathcal{F}^{s}_{T_x,C\delta}(y)\cap \mathcal{F}^{u}_{T_x,C\delta}(z)
    \end{array} \right.$$
    is a holomorphic diffeomorphism onto its image (since it is a homeomorphism onto its image and it is holomorphic along each variable because $\mathcal{F}^s_{T_{x}}$ and $\mathcal{F}^u_{T_{x}}$ are transversely holomorphic). The same can be said for $x'$.
    Therefore, by restricting $T_x$ and $T_x'$, we can assume these maps are holomorphic diffeomorphisms.
    In the following, for a smooth transversal $T$ to the flow (in the sense that the vector field $X$ generating $(\varphi^t)$ is nowhere tangent to $T$), we will denote by $T^w:=\bigcup_{-\epsilon<t<\epsilon}\varphi^t(T)$.
    Let $\epsilon>0$ small enough so that $T_x^w$ and $T^w_{x'}$ are small smooth transversals to $\mathcal{F}^{\hat{c}}$ at $x$ and $x'$ respectively. 
    We obtained thus $C^\infty$ coordinates on $T^w_x$ given by $$(t,y,z) \in \left ] - \epsilon, \epsilon \right [ \times \mathcal{F}^{u}_{\delta}(x) \times \mathcal{F}^{s}_{\delta}(x) \mapsto \varphi^t(\eta_x(y,z)) \in T^w_x,$$ and similarly on $T^w_{x'}$.
    Let $h^{\hat{c}}_{x,x'}$ denote the subcenter holonomy between $x$ and $x'$ with respect to $T^w_x$ and $T^w_{x'}$. Let $t\in \mathbb R$ small and $p=(y,z)\in T_x$. 
    Then since the subcenter foliation is invariant by the flow, it comes $h^{\hat{c}}_{x,x'}(\varphi^t(p))=\varphi^t(h^{\hat{c}}_{x,x'}(p))$ so under the above coordinates:
    $$h^{\hat{c}}_{x,x'}(t,y,z)=(t+s_{x,x'}(y,z),h^{\hat{c}, \mathcal{F}^{\hat{c}u}(x)}_{x,x'}(y),h^{\hat{c}, \mathcal{F}^{\hat{c}s}(x)}_{x,x'}(z)),$$
    where $\hat{c}, \mathcal{F}^{\hat{c}*}(x)$ in the exponent means the subcenter holonomy in the subcenter-$*$ leaf of $x$ with respect to strong $*$-leaves, and $s_{x,x'}(y,z)$ is the unique small real number $s$ so that $h^c_{x,x'}(y,z)\in \mathcal{F}_{\operatorname{loc}}^{\hat{c}}(\varphi^{-s}(y,z))$.
    Remark that $h^c_{x,x'}(y,z)=(h^{\hat{c}, \mathcal{F}^{\hat{c}u}(x)}_{x,x'}(y),h^{\hat{c}, \mathcal{F}^{\hat{c}s}(x)}_{x,x'}(z))$.
    Since the center foliation is smooth and transversely holomorphic by Proposition \ref{prop:Fcstrholo}, the above shows that the subcenter foliation has uniformly smooth holonomy if and only if the map $s_{x,x'}$ is smooth and its partial derivatives vary continuously with $(x,y,z)$.\\
    Fix $x$ and $z\in \mathcal{F}^{s}_{\operatorname{loc}}(x)$. Let $z'=h^{\hat{c}, \mathcal{F}^{\hat{c}s}(x)}_{x,x'}(z)$.
    Denote by $h^{\hat{c}, \mathcal{F}^{cu},T_x}_{z,z'}$ the subcenter holonomy in $\mathcal{F}^{cu}(z)$ between the transversals $\mathcal{F}^{wu}_{T_x,\operatorname{loc}}(z):=(\mathcal{F}^{u}_{\delta}(x)\times \{z\})^w$ and  $\mathcal{F}^{wu}_{T_{x'},\operatorname{loc}}(z'):=(\mathcal{F}^{u}_{\delta}(x')\times \{z'\})^w$.
    Let $y \in \mathcal{F}^{u}_{\operatorname{loc}}(x)=\mathcal{F}^{u}_{\delta}(x)\times \{z\}$ and $t$ a small real number. Then
    $$h^{\hat{c}, T^w_{x'},T^w_{x}}_{x,x'}(t,(y,z))=h^{\hat{c}, \mathcal{F}^{cu},T_x}_{z,z'}(t,y)=(t,0)+h^{\hat{c}, \mathcal{F}^{cu},T_x}_{z,z'}(y).$$
    On the other hand, 
    $$h^{\hat{c}, \mathcal{F}^{cu},T_x}_{z,z'}=
    \left(h^{\hat{c}, \mathcal{F}^{wu}_{\operatorname{loc}}(z'),\mathcal{F}^{wu}_{T_x,\operatorname{loc}}(z')}_{z'} \right)^{-1} \circ h^{\hat{c}, \mathcal{F}^{wu}_{\operatorname{loc}}(z'),\mathcal{F}^{wu}_{\operatorname{loc}}(z)}_{z,z'}\circ h^{\hat{c}, \mathcal{F}^{wu}_{\operatorname{loc}}(z),\mathcal{F}^{wu}_{T_x,\operatorname{loc}}(z)}_{z}.$$
    For $y' \in \mathcal{F}^{cu}_{\operatorname{loc}}(z) \simeq \Phi_{\operatorname{loc}}(z) \times \mathcal{F}^{\hat{c}u}_{\operatorname{loc}}(z) $, denote by $t_{z}(y')\in \mathbb R$ the projection of $y'$ onto the first factor.
    The map $t_{z}$ is a $C^\infty$ submersion whose differential is the projection in $E^{cu}$ on $\mathbb RX$ along $E^{\hat{c}u}$; its partial derivatives vary continuously with $(y',z)$.\\
    It comes that for $y \in \mathcal{F}^{u}_{\delta}(x)\times \{z\}$,
    $$h^{\hat{c}, \mathcal{F}^{wu}_{\operatorname{loc}}(z),\mathcal{F}^{wu}_{T_x,\operatorname{loc}}(z)}_{z}(y)=(t_z(y),h^{c, \mathcal{F}^{u}_{\operatorname{loc}}(z),\mathcal{F}^{u}_{T_x,\operatorname{loc}}(z)}_{z}(y)).$$
    Also, for small $t\in \mathbb R$ and $y'\in \mathcal{F}^u_{\operatorname{loc}}(z)$, 
    $$h^{\hat{c}, \mathcal{F}^{wu}_{\operatorname{loc}}(z'),\mathcal{F}^{wu}_{\operatorname{loc}}(z)}_{z,z'}(t,y')=(t,h^{\hat{c}, \mathcal{F}^{\hat{c}u}(z)}_{z,z'}(y')).$$
    Therefore,
    \begin{align*}
        h^{\hat{c}, \mathcal{F}^{cu},T_x}_{z,z'}(t,y)&=(t,0)+
    \left(h^{\hat{c}, \mathcal{F}^{wu}_{\operatorname{loc}}(z'),\mathcal{F}^{wu}_{T_x,\operatorname{loc}}(z')}_{z'} \right)^{-1} \left(t_z(y),h^{\hat{c}, \mathcal{F}^{\hat{c}u}(z)}_{z,z'}\circ h^{c, \mathcal{F}^{u}_{\operatorname{loc}}(z),\mathcal{F}^{u}_{T_x,\operatorname{loc}}(z)}_{z}(y) \right)\\
    &=\left(t+t_z(y)-t_{z'}(g_z(y)),g_z(y) \right)
    \end{align*}
    where $g_z(y)=\left(h^{c, \mathcal{F}^{u}_{\operatorname{loc}}(z'),\mathcal{F}^{u}_{T_x,\operatorname{loc}}(z')}_{z'}\right)^{-1}\circ h^{\hat{c}, \mathcal{F}^{\hat{c}u}(z)}_{z,z'}\circ h^{c, \mathcal{F}^{u}_{\operatorname{loc}}(z),\mathcal{F}^{u}_{T_x,\operatorname{loc}}(z)}_{z}(y)$.
    Since $g_z$ is holomorphic, the previous shows that for every such $x,z$, the map $s_{x,x'}(\cdot,z)$ is $C^\infty$ and for every $k$ its partial derivatives of order $k$ vary continuously with $(x,y,z)$.
    In the same way by reverting the roles of $y$ and $z$, for every such $x,y$, the map $s_{x,x'}(y,\cdot)$ is $C^\infty$ and for every $k$ its partial derivatives of order $k$ vary continuously with $(x,y,z)$. 
    By Journé's lemma, for every $x$ the map $s_{x,x'}$ is $C^\infty$.
    Moreover, as the proof of Journé's lemma shows (which uses Campanato's Lemma, see \cite{zavattini_regularity_nodate}), for every $k$ the partial derivatives of order $k$ of $s_{x,x'}$ vary continuously with $(x,y,z)$.
    Therefore, $\mathcal{F}^{\hat{c}}$ has uniformly $C^\infty$ holonomy, which eventually proves that it is a $C^\infty$ foliation.
\end{proof}

\subsection{The classification}
\begin{corollary}
\label{cor:classif}
    Let $(\varphi^t)$ a smooth transversely holomorphic partially hyperbolic flow on a smooth compact manifold $M$ of dimension $7$ with a flow invariant compact subcenter foliation $\mathcal{F}^{\hat{c}}$ with $C^1$ leaves and trivial holonomy.\\
    Then 
    \begin{enumerate}[label=(\roman*)]
        \item $\overline{M}$ is a smooth connected compact manifold of dimension $5$, $M \xrightarrow[]{p} \overline{M}$ is a smooth fiber bundle and $(\overline{\varphi}^t):=(\overline{\varphi^t(x)})$ is a smooth transversely holomorphic Anosov flow induced by a smooth non-vanishing vector field $\overline{X}$ ;
        \item In particular if $(\overline{\varphi}^t)$ is topologically transitive, then $(\overline{\varphi}^t)$ is either $C^\infty$-orbit equivalent to the suspension of a hyperbolic automorphism of a complex torus, or, up to finite covers, $C^\infty$-orbit equivalent to the geodesic flow on the unit tangent bundle of a compact hyperbolic three-dimensional manifold.
    \end{enumerate}
\end{corollary}
\begin{proof}
    $\overline{M}$ is compact and connected since $p$ is continuous.
    By Reeb's stability Theorem, since $\mathcal{F}^{\hat{c}}$ is a $C^\infty$ foliation with compact leaves and trivial holonomy, for every subcenter leaf $L$, there exists a small embedded transversal $T$ and a normal neighborhood $N(L)$ of $L$ in $M$ diffeomorphic to $L\times T$.
    For every $\overline{x} \in \overline{M}$, let $T_{\overline{x}}$ such a transversal given for $\mathcal{F}^{\hat{c}}(\overline{x})$. Then $p(N(L))=p(T_{\overline{x}})\subset\overline{M}$ is an open subset of $\overline{M}$ which is homeomorphic to $T_{\overline{x}}$.
    The transition map between two such local charts is given by smooth subcenter holonomies which gives makes $\overline{M}$ a $C^\infty$ manifold of dimension five.
    Therefore, $p:M\to \overline{M}$ is a smooth fiber bundle.
    The differential of $p$ at $x=(x_1,x_2) \in N(L) \cong L \times T$ is given by the projection $v=(v_1,v_2)\in T_{x_1}L\times T_{x_2}T \mapsto v_2 \in T_{\overline{x}}p(T)$. 
    The smooth flow $(\varphi^t)$ on $M$ induces naturally a non-stationary
    smooth flow $(\overline{\varphi}^t)$ on $\overline{M}$ such that for every $t \in \mathbb R$,
    $$\overline{\varphi}^t\circ p=p\circ \varphi^t.$$
    Its orbit foliation coincides with the projection of the center foliation which is transversely holomorphic by Proposition \ref{prop:Fcstrholo}, so $(\overline{\varphi}^t)$ is also a transversely holomorphic flow induced by a smooth non-vanishing vector field.\\
    We now prove that $(\overline{\varphi}^t)$ is Anosov.
    Fix a smooth Riemannian metric $\overline{g}$ on $\overline{M}$. 
    Since $E^{s}\oplus \mathbb RX \oplus E^u$ is a continuous subbundle of $TM$ transverse to $E^{\hat{c}}$,
    we can define a continuous Riemannian metric $g^0$ on $M$ such that for $x \in M$ and $v,v'\in E^{s}_x\oplus \mathbb RX(x) \oplus E^u_{x}$, $g^0_x(v,v')=\overline{g}_{\overline{x}}(d_xp(v),d_xp(v'))$.
    Now let $g$ a smooth Riemannian metric on $M$. 
    Since $M$ is compact and $g,g^0$ are continuous metrics on $M$, they are equivalent.
    Therefore, by definition of $(\overline{\varphi}^t)$, the definition of the partial hyperbolicity of $(\varphi^t)$ and the above, it comes that
    $(\overline{\varphi}^t)$ is a smooth Anosov flow with respect to $\overline{g}$ whose stable and unstable distributions are given by $dp(E^s)$ and $dp(E^u)$ respectively.\\
    The second point of the proposition is just Theorem C of \cite{abouanass_global_2025}.
\end{proof}
\begin{corollary}
    Let $(\varphi^t)$ a smooth transversely holomorphic partially hyperbolic flow on a smooth compact manifold $M$ of dimension $7$ with a flow invariant compact subcenter foliation $\mathcal{F}^{\hat{c}}$ with $C^1$ leaves and trivial holonomy. Suppose $(\varphi^t)$ is topologically transitive.\\
    Then the center-stable and center-unstable foliations are transversely projective.
    The subcenter-stable and subcenter-unstable foliations are $C^\infty$.
\end{corollary}
\begin{proof}
    The flow $(\overline{\varphi}^t)$ on $\overline{M}$ is topologically transitive because $(\varphi^t)$ is topologically transitive.
    By the previous corollary and the results of \cite{abouanass_global_2025}, the weak stable and weak unstable 
    foliations of $(\overline{\varphi}^t)$ are transversely projective, while its stable and unstable foliations are $C^\infty$.
    Therefore, since the center-stable (respectively the center-unstable) foliation is the pull-back by $p$ of the weak stable (respectively the weak unstable) foliation of $(\overline{\varphi}^t)$, it is transversely projective.
    In the same way,  since the subcenter-stable (respectively the subcenter-unstable) foliation is the pull-back by $p$ of the stable (respectively the unstable) foliation of $(\overline{\varphi}^t)$, it is $C^\infty$.
\end{proof}

\newpage
\bibliographystyle{abbrv}
\bibliography{article_4.bib}
\end{document}